\documentclass[reqno]{amsart}

\usepackage[foot]{amsaddr}
\let\oldtocsection=\tocsection
 \let\oldtocsubsection=\tocsubsection
 \let\oldtocsubsubsection=\tocsubsubsection
\renewcommand{\tocsection}[2]{\vspace{0.5em}\hspace{0em}\oldtocsection{#1}{#2}}
\renewcommand{\tocsubsection}[2]{\vspace{0.5em}\hspace{1em}\oldtocsubsection{#1}{#2}}
\renewcommand{\tocsubsubsection}[2]{\vspace{0.5em}\hspace{2em}\oldtocsubsubsection{#1}{#2}}
 \usepackage[utf8]{inputenc}
\usepackage[T1]{fontenc}
\usepackage[english]{babel}
\everymath{\displaystyle}
\usepackage{
amsmath,
amssymb,
cancel,
caption,
comment,
dsfont,
dsfont,
etoolbox,
epsfig,
epstopdf,
enumitem,
graphicx,
graphics,
geometry,
lipsum,
lastpage,
latexsym,
marginnote,
subcaption,
tikz,
tkz-tab,
url,
xcolor,
}
\usetikzlibrary{decorations.pathmorphing,patterns,scopes,intersections,calc}

 \makeatletter
\patchcmd{\@settitle}{center}{flushleft}{}{}
\patchcmd{\@settitle}{center}{flushleft}{}{}
\patchcmd{\@setauthors}{\centering}{\raggedright}{}{}
\patchcmd{\abstract}{3pc}{0pt}{}{} 
\makeatother
\makeatletter
\renewcommand*\@maketitle{%
  \normalfont\normalsize
  \@adminfootnotes
  \@mkboth{\@nx\shortauthors}{\@nx\shorttitle}%
  \global\topskip42\p@\relax 
  \@settitle
  \ifx\@empty\authors \else \@setauthors \fi
  \ifx\@empty\@date \else {\vskip 1em \vtop{\centering\large\@date\@@par}}\fi
  \ifx\@empty\@dedicatory
  \else
    \baselineskip18\p@
    \vtop{\centering{\footnotesize\itshape\@dedicatory\@@par}%
      \global\dimen@i\prevdepth}\prevdepth\dimen@i
  \fi
  \@setabstract
  \normalsize
  \if@titlepage
    \newpage
  \else
    \dimen@34\p@ \advance\dimen@-\baselineskip
    \vskip\dimen@\relax
  \fi
} 
\renewcommand*\@adminfootnotes{%
  \let\@makefnmark\relax  \let\@thefnmark\relax
  \ifx\@empty\@subjclass\else \@footnotetext{\@setsubjclass}\fi
  \ifx\@empty\@keywords\else \@footnotetext{\@setkeywords}\fi
  \ifx\@empty\thankses\else \@footnotetext{%
    \def\par{\let\par\@par}\@setthanks}%
  \fi
}
\makeatother
\usepackage{fancyhdr}
\def\nline{\\ \noalign{\medskip}}

\newcounter{dummy}
\numberwithin{dummy}{section}
\newtheorem{Theorem}[dummy]{Theorem}
\newtheorem{Corollary}[dummy]{Corollary}
\newtheorem{Definition}[dummy]{Definition}
\newtheorem{Lemma}[dummy]{Lemma}
\newtheorem{Proposition}[dummy]{Proposition}

\numberwithin{equation}{section}
\newcommand\xqed[1]{\leavevmode\unskip\penalty9999 \hbox{}\nobreak\hfill\quad\hbox{#1}}
\RequirePackage[bookmarks,%
                colorlinks,%
                urlcolor=blue,%
                citecolor=blue,%
                linkcolor=blue,%
                hyperfigures,%
                pagebackref=false,%
                pdfcreator=LaTeX,%
                breaklinks=true,%
                pdfpagelayout=SinglePage,%
                bookmarksopen=true,%
                bookmarksopenlevel=2]{hyperref}
\begin{document}
\title[\fontsize{7}{9}\selectfont  ]{Well-Posedness and Polynomial energy decay rate of a transmission problem for Rayleigh beam model with heat conduction}
\author{Mohammad Akil$^{1,\ast}$, Mouhammad Ghader$^{2}$, Zayd Hajjej$^{3}$  and Mohamad Ali sammoury$^{2}$  \vspace{0.5cm}\\
$^1$Univ. Polytechnique  Hauts-de-France, INSA Hauts-de-France,  CERAMATHS-Laboratoire de Mat\'eriaux C\'eramiques et de Math\'ematiques, F-59313 Valenciennes, France.\vspace{0.5cm}\\
$^2$ Al Maaref University, Department of Mathematics, Beirut Campus, Airport Bridge, Ghobeiry 5078 0025, LEBANON.\vspace{0.5cm}\\
$^3$ Department of Mathematics, College of Science, King Saud University, P.O. Box 2455, Riyadh 11451, Saudi Arabia.\vspace{0.5cm}\\
Email: mohammad.akil@uphf.fr, mouhammad.ghader@mu.edu.lb, zhajjej@ksu.edu.sa, mohamad.sammoury@mu.edu.lb\\ \\
$^\ast$ Corresponding author. \vspace{0.5cm}\\}
\setcounter{equation}{0}
\begin{abstract}
 In this paper, we investigate the stability of the transmission problem for Rayleigh beam model with heat conduction. First, we reformulate our system into an evolution equation and prove our problem's well-posedness. Next, we demonstrate the resolvent of the operator is compact in the energy space, then by using the general criteria of Arendt-Batty, we prove that the thermal dissipation is enough to stabilize our model. Finally, a polynomial energy decay rate has been obtained which depends on the mass densities and the moments of inertia of the Rayleigh beams.
\\[0.1in]
\end{abstract}
\pagenumbering{roman}
\maketitle
\tableofcontents

\clearpage
\pagenumbering{arabic}
\setcounter{page}{1}
\section{Introduction}\label{S:=1}
\noindent In this paper, we study the stability of  a transmission problem for Rayleigh beam model with  heat conduction
\begin{eqnarray}
\rho_1\, u_{tt}-\alpha_1\, u_{xxtt}+\beta_1 u_{xxxx}+\gamma\, \theta_{xx}=0,& (x,t)\in (0,L_0)\times (0,+\infty),\label{E:=(1.1)}\nline
\rho_2\, y_{tt}-\alpha_2\, y_{xxtt}+\beta_2\, y_{xxxx}=0,& (x,t)\in (L_0,L)\times (0,+\infty),\label{E:=(1.2)}
\nline
\rho_0\,\theta_t-\kappa\, \theta_{xx}-\gamma\, u_{xxt}=0,& (x,t)\in (0,L_0)\times (0,+\infty),\label{E:=(1.3)}
\end{eqnarray}
with  boundary conditions
\begin{eqnarray}
\theta(0,t)=\theta(L_0,t)=0,&t\in  (0,+\infty),\label{E:=(1.4)}\nline
u(0,t)=u_x(0,t)=0,&t\in  (0,+\infty),\label{E:=(1.5)}\nline
y(L,t)=y_x(L,t)=0,&t\in  (0,+\infty),\label{E:=(1.6)}
\end{eqnarray}
 transmission conditions
\begin{eqnarray}
u(L_0,t)=y(L_0,t),&t\in  (0,+\infty),\label{E:=(1.7)}\nline
u_x(L_0,t)=y_x(L_0,t),&t\in  (0,+\infty),\label{E:=(1.8)}\nline
\beta_1 u_{xx}(L_0,t)=\beta_2\, y_{xx}(L_0,t),&t\in  (0,+\infty),\label{E:=(1.9)}\nline
\gamma\, \theta_x(L_0,t)+\beta_1\, u_{xxx}(L_0,t)-\alpha_1 u_{xtt}(L_0,t)-\beta_2\, y_{xxx}(L_0,t)+\alpha_2 y_{xtt}(L_0,t)=0,&t\in  (0,+\infty),\label{E:=(1.10)}
\end{eqnarray}
and  initial data
\begin{eqnarray}
 \left(u(x,0),u_t(x,0),\theta(x,0)\right)=\left(u_0(x),u_1(x),\theta_0(x)\right),& x\in (0,L_0),\label{E:=(1.11)}\nline
 \left(y(x,0),y_t(x,0)\right)=\left(y_0(x),y_1(x)\right),& x\in (L_0,L),\label{E:=(1.12)}
\end{eqnarray}
where, for $i=1, 2$,  $\rho_i>0$ is the mass density per unit volume, $\alpha_i>0$ is the moment of inertia of the cross-sections, $\beta_i>0$ is the stiffness constant, while $\rho_0>0$ and $\kappa>0$ represent, respectively, the specific heat and the thermal conductivity. Here $0<L_0<L$ and $\gamma$ is a non-zero real number.\\
The model at hand describes a Rayleigh beam formed of two distinct materials, one of which is sensitive to thermal differences and the other of which is unaffected by temperature changes. In other words, the material has a limited thermoelastic effect \cite{Rivera-Oquendo-2001, Rivera-Oquendo-2004}.
\\
The stabilization of the Rayleigh beam equation retains the attention of many authors. In this regard, different types of damping have been introduced to the Rayleigh beam equation and several uniform and polynomial stability results have been obtained. Rao \cite{Rao-Rayeligh-1996} studied the stabilization of Rayleigh beam equation subject to a positive internal viscous damping. Using a constructive approximation, he established the optimal exponential decay rate. There exists  many papers concerning the stability with different types of damping \cite{MNSW-2017,BMW-2015,Wehbe2006,MR953313,MR1029069,MR1061153,MR1108855}.\\
In \cite{MR3395754}, the authors  are concerned with the stability of an interconnected system of an Euler-Bernoulli beam and a heat equation with boundary coupling. The boundary temperature of the beam is fed as the boundary moment of the Euler-Bernoulli equation and the boundary angular velocity of the Euler-Bernoulli beam is fed into the boundary heat flux of the heat equation. It is shown that the spectrum of the closed-loop system consists of only two branches: one along the real axis and the other along two parabolas that are symmetric to the real axis and open to the imaginary axis. The asymptotic expressions of both eigenvalues and eigenfunctions are obtained. With a careful estimate of the resolvent operator, the completeness of the root subspaces of the system is verified. The Riesz basis property and exponential stability of the system are then proved. Moreover, it is shown that the semigroup generated by the system operator is of Gevrey class $\delta>2$.\\
In \cite{MR3163488}, the authors  studied the stabilization problem for a coupled PDE system in which the beam ($1-$dimensional or $2-$dimensional) and heat equations are coupled at the boundary conditions. Moreover, a dissipative damping is produced in the heat equation via the boundary connections only.
   In the first part, the authors considered the asymptotic behavior of the $1-$dimensional coupled system mainly by the Riesz basis approach. By using a detailed spectral analysis for the system operator, they obtained asymptotic expressions for the spectrum and the corresponding eigenvectors. The authors further obtained a spectrum-determined growth condition by showing the Riesz basis property of the eigenvectors. Then, based on the spectral distribution, they deduced the Gevrey regularity of the semigroup for the system and the exponential decay rate of the system energy.
   In the second part, the authors investigated the asymptotic behavior of the $2-$dimensional coupled PDE system by using the frequency domain method. By estimating the uniform boundedness of the norm of the resolvent operator along the imaginary axis, they showed that the $2-d$ coupled system is also exponentially stable when an additional dissipation in the boundary of the plate part exists.\\
 We mention some papers studied the stability of different system under heat conduction \cite{Rivera-Racke-2021,Rivera-Racke-2017,Rivera-Racke-2001,MR4500789,ARY-2023,SHLR-2019,NR-2018,LA1,LA2,LA3,LA4,LA5,LA6,LA7,LA8,LA9,LA10,Gues2022,Guess2022,YMG-2019,Ammari-2016,Ammari-2015,NVF-2009,WN-2014,AGW-2018,BenAissa2022,DSPS-2022,ACIY-2019,WY-2020,WY-2022,FYA-2023}.

 Now, we mention some papers concerning a transmission wave-heat system. In \cite{MR3263148}, the author studied the stability analysis of an interaction system comprised of a wave equation and a heat equation with memory, where the hereditary heat conduction is due to Gurtin-Pipkin law or Coleman-Gurtin law. First, she showed the strong asymptotic stability of solutions to this system. Then, the exponential stability of the interaction system is obtained when the hereditary heat conduction is of Gurtin-Pipkin type. Further, she showed the lack of uniform decay of the interaction system when the heat conduction law is of Coleman-Gurtin type. In \cite{MR4491743}, the authors extended the result of \cite{MR3263148} by proving the optimal polynomial decay rate of type $1/t$ when the heat conduction law is of Coleman-Gurtin type. \\
 \\
 To our best knowledge, the transmission problem for Rayleigh beam with heat conduction is not treated in the literature. The goal of this paper is to fix this gap by considering System \eqref{E:=(1.1)}-\eqref{E:=(1.12)}. \\
 \\
 The paper is organized as follows: In Section \ref{S:=2}, we formulate the System \eqref{E:=(1.1)}-\eqref{E:=(1.12)}  into an evolution equation $\Phi_t=\mathcal{A}\Phi,\quad \Phi(0)=\Phi_0=(u_0,y_0,u_1,y_1,\theta_0)$ (see \eqref{E:=(2.14)}). Next,  Section \ref{S:=3} is divided into two subsections. In subsection \ref{S:=3.1} we study the well-posedness of Problem \eqref{E:=(1.1)}-\eqref{E:=(1.12)}\textcolor{red}{.} According to Lumer-Phillips theorem (see \cite{Liu-Zheng-1999, Pazy-1983}), we  prove that the operator $\mathcal{A}$ is m-dissipative. In Subsection \ref{S:=3.2}, we prove the strong stability of \eqref{E:=(1.1)}-\eqref{E:=(1.12)}. Firstly, we prove that the the operator ${\mathcal{A}}$  has a compact resolvent on the energy space. Next, we prove the strong stability of System \eqref{E:=(1.1)}-\eqref{E:=(1.12)} by using Arendt-Batty Theorem. In Section  \ref{S:=4}, we prove the polynomial stability of System \eqref{E:=(1.1)}-\eqref{E:=(1.12)}. The decay rate of the energy depends on the physical coefficients. We obtain the following result:\\
 $\bullet$ A polynomial energy decay rate of type $t^{-2}$ if $\rho_1\geq \rho_2$ and $\alpha_1\geq \alpha_2$. \\
 $\bullet$ A polynomial energy decay rate of type $t^{-1}$ if $\rho_1<\rho_2$ or $\alpha_1<\alpha_2$.\\
 We use Borichev-Tomilov Theorem combining with a specific multiplier technics and a particular attention of the sharpness of the estimates to optimize the results.
\section{Formulation of the Problem} \label{S:=2}
\noindent We start this section by  defining the energy of a solution of System  \eqref{E:=(1.1)}-\eqref{E:=(1.12)} by
\begin{equation*}
E(t)=\frac{1}{2}\int_0^{L_0}\left( \rho_1 \left|u_t\right|^2+\alpha_1\, \left|u_{xt}\right|^2+\beta_1\left|u_{xx}\right|^2+\rho_0\left|\theta\right|^2\right) dx+\frac{1}{2}\int_{L_0}^L\left( \rho_2 \left|y_t\right|^2+\alpha_2\, \left|y_{xt}\right|^2+\beta_2\left|y_{xx}\right|^2\right) dx.
\end{equation*}
Multiplying \eqref{E:=(1.1)} and \eqref{E:=(1.3)} by $u_t$ and $\theta$, respectively,    integrating by parts over $(0,L_0)$ with respect to $x$ and taking the sum of the resulting equations, we get
\begin{equation}\label{E:=(2.1)}
\begin{array}{ll}

 \frac{1}{2}\displaystyle{\frac{d}{ dt}\int_0^{L_0}\left(\rho_1 \left|u_t\right|^2  +\alpha_1\, \left|u_{xt}\right|^2+\beta_1 \left|u_{xx}\right|^2+\rho_0\left|\theta\right|^2\right)dx+\kappa \int_0^{L_0} \left|\theta_x\right|^2dx}

\nline

\hspace{2cm}\displaystyle{+\left[\left(\gamma\, \theta_x+\beta_1\,  u_{xxx}-\alpha_1 \, u_{xtt}\right) u_t-\beta_1 \, u_{xx}\, u_{xt} -\gamma\, u_{xt}\,  \theta\right]_0^{L_0}=0}.

\end{array}
\end{equation}
Next, multiplying \eqref{E:=(1.2)}  by $y_t$,     integrating by parts over $(L_0,L)$ with respect to $x$,  we obtain
\begin{equation}\label{E:=(2.2)}
 \frac{1}{2}\frac{d}{dt}\int_{L_0}^L\left(\rho_2 \left|y_t\right|^2  +\alpha_2\, \left|y_{xt}\right|^2+\beta_2 \left|y_{xx}\right|^2\right)dx+\left[\left(\beta_2\,  y_{xxx}-\alpha_2\,  y_{xtt}\right) y_t-\beta_2 \, y_{xx}\, y_{xt}\right]_{L_0}^L=0.
\end{equation}
Adding \eqref{E:=(2.1)} and \eqref{E:=(2.2)}, then using the boundary condition \eqref{E:=(1.4)}-\eqref{E:=(1.10)}, we infer that
\begin{equation*}
E'(t)=-\kappa  \int_0^{L_0}\left|\theta_x\right|^2dx\leq 0.
\end{equation*}
Hence, System \eqref{E:=(1.1)}-\eqref{E:=(1.12)} is dissipative in the sense that its energy is non increasing with respect to the time $t.$ \\[0.1in]
\noindent We start our study by formulating  problem \eqref{E:=(1.1)}-\eqref{E:=(1.12)} in an appropriate Hilbert space: \\[0.1in]
$\bullet$ We introduce the following  spaces:
\begin{equation*}
\left\{
\begin{array}{llll}

H^1_L\left(0,L_0\right)=\left\{u\in H^1\left(0,L_0\right)\ |\ u(0)=0\right\},

\nline

H^1_R\left(L_0,L\right)=\left\{y\in H^1\left(L_0,L\right)\ |\ y(L)=0\right\},

\nline

\mathbb{V}_1=\left\{\left(u,y\right)\in H^1_L(0,L_0)\times H^1_R(L_0,L)\ |\  u(L_0)=y(L_0)
\right\},

\nline

\mathbb{W}_1=\left\{\left(u,y\right)\in \left(H^2(0,L_0)\times H^2(L_0,L)\right)\cap \mathbb{V}_1\ |\
u_x(0)=y_x(L)=0,\  u_x(L_0)=y_x(L_0)\right\}.

\end{array}
\right.
\end{equation*}
Set
\begin{equation*}
\left\{
\begin{array}{lll}

\mathbb{L}_1=L^2(0,L_0)\times L^2(L_0,L),\quad  \mathbb{L}_2=L^2(0,L_0)\times L^2(L_0,L)\times L^2(0,L_0),

\nline

 \mathbb{V}_2=\mathbb{V}_1\times L^2(0,L_0),\quad \mathbb{W}_2=\mathbb{W}_1\times H^1_0(0,L_0).

\end{array}
\right.
\end{equation*}
$\bullet$ Let $\left(u,y,\theta\right)$  be a regular solution of System  \eqref{E:=(1.1)}-\eqref{E:=(1.12)}. Let $\left(\hat{u},\hat{y}, \hat{\theta}\right)\in {\mathbb{W}_2}$. Multiplying \eqref{E:=(1.1)}, \eqref{E:=(1.2)},  and \eqref{E:=(1.3)} by $\overline{\hat{u}},\ \overline{\hat{y}},$ and $\overline{\hat{\theta}}$, respectively, integrating  by parts over $(0,L_0),\ (L_0,L),$ and $(0,L_0)$, respectively and  then taking the sum, we derive
\begin{equation}\label{E:=(2.3)}
\begin{array}{llll}

\int_0^{L_0} \left(\rho_1 u_{tt}\overline{\hat{u}}+\alpha_1\, u_{xtt}\overline{\hat{u}_x}\right)dx+\int_{L_0}^L \left(\rho_2 y_{tt}\overline{\hat{y}}+\alpha_2\, y_{xtt}\overline{\hat{y}_x}\right)dx-\gamma\, \int_0^{L_0}\left(\theta_{x}\overline{\hat{u}_x}-u_{xt}\overline{\hat{\theta_x}}\right) dx

\nline

+\beta_1 \int_0^{L_0}u_{xx}\overline{\hat{u}_{xx}}\, dx+\beta_2 \int_{L_0}^Ly_{xx}\overline{\hat{y}_{xx}}\, dx+\rho_0\int_0^{L_0}\theta_t \overline{\hat{\theta}}\, dx+\kappa\, \int_0^{L_0}\theta_{x}\overline{\hat{\theta}_x}\, dx

\nline

-\left[\left(\alpha_1 u_{xtt}-\beta_1 u_{xxx}\right) \overline{\hat{u}}+\beta_1 u_{xx} \overline{\hat{u}_x}+\kappa  \theta_{x} \overline{\hat{\theta}}\right]_{0}^{L_0}-\left[\left(\alpha_2 y_{xtt}-\beta_2 y_{xxx} \right) \overline{\hat{y}}+\beta_2 y_{xx} \overline{\hat{y}_x}\right]_{L_0}^{L}

\nline

+\gamma\left[\theta_{x}\overline{\hat{u}}-u_{xt}\overline{\hat{\theta}}\right]_0^{L_0}=0.

\end{array}
\end{equation}
Since $\left(\hat{u},\hat{y}, \hat{\theta}\right)\in {\mathbb{W}_2}$, then
\begin{equation*}
\hat{u}(0)=\hat{u}_x(0)=\hat{y}(L)=\hat{y}_x(L)=\hat{\theta}(0)=\hat{\theta}(L_0)=0,\quad \hat{u}(L_0)=\hat{y}(L_0),\quad  \hat{u}_x(L_0)=\hat{y}_x(L_0).
\end{equation*}
Using the above boundary conditions in \eqref{E:=(2.3)}, we get
\begin{equation*}
\begin{array}{llll}

\int_0^{L_0} \left(\rho_1 u_{tt}\overline{\hat{u}}+\alpha_1\, u_{xtt}\overline{\hat{u}_x}\right)dx+\int_{L_0}^L \left(\rho_2y_{tt}\overline{\hat{y}}+\alpha_2\, y_{xtt}\overline{\hat{y}_x}\right)dx+\gamma\, \int_0^{L_0}\left(u_{xt}\overline{\hat{\theta}_x}-\theta_{x}\overline{\hat{u}_x}\right) dx

\nline

+\beta_1 \int_0^{L_0}u_{xx}\overline{\hat{u}_{xx}}\, dx+\beta_2 \int_{L_0}^Ly_{xx}\overline{\hat{y}_{xx}}\, dx

+\rho_0\int_0^{L_0}\theta_t \overline{\hat{\theta}}\, dx+\kappa\, \int_0^{L_0}\theta_{x}\overline{\hat{\theta}_x}\, dx

\nline

+\left(\beta_2 y_{xx}\left(L_0\right)-\beta_1 u_{xx}\left(L_0\right)\right) \overline{\hat{y}_x}\left(L_0\right)

\nline

+\left(\gamma \theta_x(L_0)+\beta_1 u_{xxx}\left(L_0\right)-\alpha_1 u_{xtt}\left(L_0\right) -\beta_2 y_{xxx}\left(L_0\right)+\alpha_2 y_{xtt}\left(L_0\right) \right) \overline{\hat{y}}\left(L_0\right)=0.

\end{array}
\end{equation*}
Using the boundary conditions  \eqref{E:=(1.9)} and \eqref{E:=(1.10)} in the above equation, we obtain
\begin{equation*}
\begin{array}{llll}

\int_0^{L_0} \left(\rho_1u_{tt}\overline{\hat{u}}+\alpha_1\, u_{xtt}\overline{\hat{u}_x}\right)dx+\int_{L_0}^L \left(\rho_2y_{tt}\overline{\hat{y}}+\alpha_2\, y_{xtt}\overline{\hat{y}_x}\right)dx+\gamma\, \int_0^{L_0}\left(u_{xt}\overline{\hat{\theta}_x}-\theta_{x}\overline{\hat{u}_x}\right) dx

\nline

+\beta_1 \int_0^{L_0}u_{xx}\overline{\hat{u}_{xx}}\, dx+\beta_2 \int_{L_0}^Ly_{xx}\overline{\hat{y}_{xx}}\, dx+\rho_0\int_0^{L_0}\theta_t \overline{\hat{\theta}}\, dx+\kappa\, \int_0^{L_0}\theta_{x}\overline{\hat{\theta}_x}\, dx=0.

\end{array}
\end{equation*}
Equivalently, the variational equation of problem \eqref{E:=(1.1)}-\eqref{E:=(1.12)} is given by
\begin{equation}\label{E:=(2.4)}
\begin{array}{llll}

\beta_1 \int_0^{L_0}u_{xx}\overline{\hat{u}_{xx}}\, dx+\beta_2 \int_{L_0}^Ly_{xx}\overline{\hat{y}_{xx}}\, dx

\nline

+\kappa\, \int_0^{L_0}\theta_{x}\overline{\hat{\theta}_x}\, dx+\gamma \int_0^{L_0}\left(u_{xt}\overline{\hat{\theta}_x}-\theta_{x}\overline{\hat{u}_x}\right) dx

\nline

+\int_0^{L_0} \left(\rho_1 u_{tt}\overline{\hat{u}}+\alpha_1\, u_{xtt}\overline{\hat{u}_x}\right)dx+\int_{L_0}^L \left(\rho_2y_{tt}\overline{\hat{y}}+\alpha_2\, y_{xtt}\overline{\hat{y}_x}\right)dx+\rho_0\int_0^{L_0}\theta_t \overline{\hat{\theta}}\, dx=0.

\end{array}
\end{equation}
$\bullet$ We identify ${\mathbb{L}_1}$ with its dual $\mathbb{L}_1'$ and ${\mathbb{L}_2}$ with its dual $\mathbb{L}_2'$,  so that we have the following continuous embeddings:
\begin{equation}\label{E:=(2.5)}
\left\{
\begin{array}{lll}

{\mathbb{W}_2}\subset {\mathbb{V}_2}\subset {\mathbb{L}_2}\subset \mathbb{V}_2'\subset \mathbb{W}_2',

\nline

{\mathbb{W}_1}\subset {\mathbb{V}_1}\subset {\mathbb{L}_1}\subset \mathbb{V}_1'\subset \mathbb{W}_1'.
\end{array}
\right.
\end{equation}
$\bullet$  We introduce the following bilinear forms:
\begin{equation}\label{E:=(2.6)}
\begin{array}{lllll}
\begin{array}{lll}

\text{for } Z=\left(u,y\right),\, \hat{Z}=\left(\hat{u},\hat{y}\right)\in \mathbb{W}_1:

\nline

a\left(Z,\hat{Z}\right)=\beta_1\left<u_{xx},\hat{u}_{xx}\right>_{L^2(0,L_0)}+\beta_2 \left<y_{xx},\hat{y}_{xx}\right>_{L^2(L_0,L)},
\end{array}
\\ \\
\begin{array}{lll}

\text{for }  \Phi=\left(u,y,\theta \right),\, \hat{\Phi}=\left(\hat{u},\hat{y},\hat{\theta}\right)\in {\mathbb{W}_2}:

\nline

b\left(\Phi,\hat{\Phi}\right)=\gamma\int_0^{L_0}\left(u_{x}\overline{\hat{\theta}_x}-\theta_{x}\overline{\hat{u}_x}\right) dx+\kappa \left<\theta_{x},\hat{\theta}_{x}\right>_{L^2(0,L_0)},
\end{array}
\\ \\
\begin{array}{lll}

\text{for }  U=\left(u,y,\theta\right),\, \hat{U}=\left(\hat{u},\hat{y},\hat{\theta}\right)\in {\mathbb{V}_2}:

\nline

c\left(U,\hat{U}\right)=\rho_1\left<u,\hat{u}\right>_{L^2(0,L_0)}+\alpha_1 \left<u_x,\hat{u}_x\right>_{L^2(0,L_0)}+\rho_2\left<y,\hat{y}\right>_{L^2(L_0,L)}\nline\hspace{2cm}+\alpha_1 \left<y_x,\hat{y}_x\right>_{L^2(L_0,L)}+\rho_0\left<\theta,\hat{\theta}\right>_{L^2(0,L_0)}.
\end{array}
\end{array}
\end{equation}
Here and below, $\left<\cdot,\cdot\right>_{L^2(0,L_0)}$ and $\left<\cdot,\cdot\right>_{L^2(L_0,L)}$ denote the usual inner product of $L^2(0, L_0)$ and $L^2(L_0, L)$, respectively,  and   $\left\|\cdot \right\|_{L^2(0,L_0)}$ and $\left\|\cdot \right\|_{L^2(L_0,L)}$   their  corresponding norms. The form $a(\cdot, \cdot)$ (resp. $c(\cdot, \cdot)$) is a bilinear continuous coercive form on $\mathbb{W}_1\times \mathbb{W}_1$ (resp. on ${\mathbb{V}_2}\times {\mathbb{V}_2}$), while $b(\cdot, \cdot)$ is a bilinear continuous form on ${\mathbb{W}_2}\times {\mathbb{W}_2}$ and satisfies
\begin{equation}\label{E:=(2.7)}
 \Re\left\{ b\left(\Phi,{\Phi}\right)\right\}=\kappa \left<\theta_{x},{\theta}_{x}\right>_{L^2(0,L_0)}=\kappa\int_0^{L_0}\left|\theta_x\right|^2 dx,\quad \forall\; \Phi=\left(u,y,\theta \right)\in {\mathbb{W}_2}.
\end{equation}
$\bullet$ We define the operators $C\in \mathcal{L}\left({\mathbb{V}_2},\mathbb{V}_2'\right)$, $B\in \mathcal{L}\left({\mathbb{W}_2},\mathbb{W}_2'\right)$, and $A_0\in \mathcal{L}\left(\mathbb{W}_1,\mathbb{W}_1'\right)$ by
\begin{equation}\label{E:=(2.8)}
\left\{
\begin{array}{lll}
\left<C{U},\hat{U}\right>_{\mathbb{V}_2'\times {\mathbb{V}_2}}:=c\left(U,\hat{U}\right),\qquad \forall\; {U}=\left(u,y,\theta\right),\ \hat{U}=\left(\hat{u},\hat{y}, \hat{\theta}\right)\in {\mathbb{V}_2},

\nline

\left<B{U},\hat{U}\right>_{\mathbb{W}_2'\times {\mathbb{W}_2}}=b\left(U,\hat{U}\right),\qquad \forall\; {U}=\left(u,y,\theta\right),\ \hat{U}=\left(\hat{u},\hat{y}, \hat{\theta}\right)\in {\mathbb{W}_2},

\nline

\left<A_0{Z},\hat{Z}\right>_{\mathbb{W}_1'\times \mathbb{W}_1}=a\left(Z,\hat{Z}\right)
,\qquad \forall\; {Z}=\left(u,y\right),\ \hat{Z}=\left(\hat{u},\hat{y}\right)\in \mathbb{W}_1,

\nline

A_1 Z=\left(A_0 Z,0\right),\quad \forall Z=(u,y)\in \mathbb{W}_1.

\end{array}
\right.
\end{equation}
The operator $C$ (resp. $A_0$) is an isomorphism of ${\mathbb{V}_2}$ onto $\mathbb{V}_2'$  (resp. $\mathbb{W}_1$ onto $\mathbb{W}_1'$) and is the canonical isomorphism, so we can introduce $c\left(\cdot,\cdot\right)$ (resp. $a\left(\cdot,\cdot\right)$ as a scalar product on ${\mathbb{V}_2}$ (resp. on $\mathbb{W}_1$), i.e.,
\begin{equation}\label{E:=(2.9)}
\left\{
\begin{array}{lll}
\|U\|_{{\mathbb{V}_2}}^2=\left<{U},\hat{U}\right>_{{\mathbb{V}_2}}=c\left(U,\hat{U}\right)=\left<C{U},\hat{U}\right>_{\mathbb{V}_2'\times {\mathbb{V}_2}},\qquad \forall\; {U}=\left(u,y,\theta\right),\ \hat{U}=\left(\hat{u},\hat{y}, \hat{\theta}\right)\in {\mathbb{V}_2},

\nline

\|Z\|_{{\mathbb{W}_1}}^2=\left<{Z},\hat{Z}\right>_{\mathbb{W}_1}=a\left(Z,\hat{Z}\right)=\left<A_0{Z},\hat{Z}\right>_{\mathbb{W}_1'\times \mathbb{W}_1},\qquad \forall\; {Z}=\left(u,y\right),\ \hat{Z}=\left(\hat{u},\hat{y}\right)\in \mathbb{W}_1.

\end{array}
\right.
\end{equation}
$\bullet$ The variational equation \eqref{E:=(2.4)} can be written in terms of the above operators
as an equation in $\mathbb{W}_2'$ as follows:
\begin{equation}\label{E:=(2.10)}
C\left(u_{tt},y_{tt},\theta_t\right)+B\left(u_t,y_t,\theta\right)+A_1(u,y)=0\quad \text{in } \mathbb{W}_2'.
\end{equation}
Furthermore, assume that $B\left(u_t,y_t,\theta\right)+A_1(u,y)\in \mathbb{V}_2'$ , then we obtain that
\begin{equation*}
\left(u_{tt},y_{tt},\theta_t\right)+C^{-1}\left(B\left(u_t,y_t,\theta\right)+A_1(u,y)\right)=0\quad \text{in } \mathbb{V}_2.
\end{equation*}
Defining $v=u_t$ and $z=y_t$, then \eqref{E:=(2.10)} can be written as
\begin{equation*}
\left(v,z,\theta\right)_t=-C^{-1}\left(B\left(v,z,\theta\right)+A_1(u,y)\right).
\end{equation*}
$\bullet$ We introduce the following energy space:
\begin{equation*}
\mathcal{H}=\mathbb{W}_1\times {\mathbb{V}_2}.
\end{equation*}
For all $\Phi=\left(\Phi_1,\Phi_2\right)\in \mathcal{H}$ and $\hat{\Phi}=\left(\hat{\Phi}_1,\hat{\Phi}_2\right)\in \mathcal{H}$, such that $\Phi_1=\left(u,y\right),\ \Phi_2=\left(v,z,\theta\right),\ \hat{\Phi}_1=\left(\hat{u},\hat{y}\right),$ and $\hat{\Phi}_2=\left(\hat{v},\hat{z},\hat{\theta}\right)$, it is easy to check that the space $ \mathcal{H}$  is a Hilbert space over $\mathbb{C}$ equipped  with the following inner product
\begin{equation}\label{E:=(2.11)}
\begin{array}{lllll}

\left<\Phi,\hat{\Phi}\right>_{\mathcal{H}}&=&\left<\Phi_1,\hat{\Phi}_1\right>_{\mathbb{W}_1}+\left<\Phi_2,\hat{\Phi}_2\right>_{{\mathbb{V}_2}}

\nline

&=&a\left(\Phi_1,\hat{\Phi}_1\right)+c\left(\Phi_2,\hat{\Phi}_2\right)

\nline

&=&	\beta_1\int_0^{L_0}u_{xx}\overline{\hat{u}_{xx}}\, dx+\beta_2\int_{L_0}^Ly_{xx}\overline{\hat{y}_{xx}} \, dx

\nline

&& \  +\int_0^{L_0}\left(\rho_1 v\overline{\hat{v}}+\alpha_1 v_{x}\overline{\hat{v}_{x}}\right) dx+\int_{L_0}^L\left(\rho_2z\overline{\hat{z}} +\alpha_2 z_{x}\overline{\hat{z}_{x}} \right) dx+\rho_0\int_{0}^{L_0}\theta\overline{\hat{\theta}}\, dx.
	
\end{array}	
\end{equation}
Hereafter, we use $\|U\|_{{\mathcal{H}}}$ to denote the corresponding norm.\\[0.1in]
$\bullet$ For all $\Phi=\left(\Phi_1,\Phi_2\right)\in \mathcal{H}$, such that $\Phi_1=\left(u,y\right)$ and $\Phi_2=\left(v,z,\theta\right),$ we define the unbounded linear operator ${\mathcal{A}}:\ D\left(\mathcal{A}\right)\subset\mathcal{H}\to \mathcal{H}$ by
\begin{equation}\label{E:=(2.13)}
{\mathcal{A}}\Phi=\left(v,z,-C^{-1}\left(B\Phi_2+A_1\Phi_1\right)\right),
\end{equation}
with domain
\begin{equation}\label{E:=(2.12)}
D\left(\mathcal{A}\right)=\left\{\Phi=\left(\Phi_1,\Phi_2\right)\in \mathbb{W}_1\times {\mathbb{V}_2}\ |\ \left(\Phi_1,\Phi_2\right)\in \mathbb{W}_1\times {\mathbb{W}_2},\  B\Phi_2+A_1\Phi_1\in  \mathbb{V}_2' \right\}.
\end{equation}
$\bullet$ If $\Phi=\left(u,y,v,z,\theta\right)$ is a regular solution of System  \eqref{E:=(1.1)}-\eqref{E:=(1.12)}, then we rewrite this system as the following evolution equation
\begin{equation}\label{E:=(2.14)}
\Phi_t={\mathcal{A}}\Phi,\quad
\Phi(0)=\Phi_0,
\end{equation}
where $\Phi_0=\left(u_0,y_0,u_1,y_1,\theta_0\right)$.
\section{Well-Posedness and Strong Stability}\label{S:=3}
\subsection{Well-posedness of the problem.}\label{S:=3.1}
\noindent For the well-posedness of Problem \eqref{E:=(1.1)}-\eqref{E:=(1.12)}, according to Lumer-Phillips theorem (see \cite{Liu-Zheng-1999, Pazy-1983}), we need to prove that the operator $\mathcal{A}$ is m-dissipative. Hence, we shall prove the following proposition.
			\begin{Proposition}\label{P:=3.1}
	The unbounded linear operator $\mathcal{A}$ is m-dissipative in the energy space $\mathcal{H}$.
		\end{Proposition}
\begin{proof} We first prove that $\mathcal{A}$ is monotone. For this aim, let $\Phi=\left(\Phi_1,\Phi_2\right)\in \mathcal{H}$, such that $\Phi_1=\left(u,y\right)$ and $\Phi_2=\left(v,z,\theta\right),$ using the definitions \eqref{E:=(2.9)}, \eqref{E:=(2.11)}, \eqref{E:=(2.12)}, and \eqref{E:=(2.13)}, we have
\begin{equation*}
\begin{array}{lll}
\left<\mathcal{A}\Phi,\Phi\right>_{\mathcal{H}}

&=&\left<\left(v,z,-C^{-1}\left(B\Phi_2+A_1\Phi_1\right)\right),\left(\Phi_1,\Phi_2\right)\right>_{\mathcal{H}}

\nline

&=&\left<\left(v,z\right),\Phi_1\right>_{\mathbb{W}_1}-\left<C^{-1}\left(B\Phi_2+A_1\Phi_1\right),\Phi_2\right>_{{\mathbb{V}_2}}

\nline

&=&a\left(\left(v,z\right),\Phi_1\right)-\left<B\Phi_2+A_1\Phi_1,\Phi_2\right>_{\mathbb{V}_2'\times {\mathbb{V}_2}}.

\end{array}
\end{equation*}
Since $\Phi\in D(\mathcal{A})$; i.e., $\Phi_2\in  {\mathbb{W}_2}$ and $\Phi_1\in  \mathbb{W}_1$, then using \eqref{E:=(2.5)}, the last equation of \eqref{E:=(2.8)}, and second-third equations of \eqref{E:=(2.8)}  in the above equation, we obtain
\begin{equation*}
\begin{array}{lll}
\left<\mathcal{A}\Phi,\Phi\right>_{\mathcal{H}}

&=&a\left(\left(v,z\right),\Phi_1\right)-\left<B\Phi_2-A_1\Phi_1,\Phi_2\right>_{\mathbb{W}_2'\times {\mathbb{W}_2}}

\nline

&=&a\left(\left(v,z\right),\Phi_1\right)-\left<B\Phi_2,\Phi_2\right>_{\mathbb{W}_2'\times {\mathbb{W}_2}}-\left<A_0\Phi_1,\left(v,z\right)\right>_{\mathbb{W}_1'\times \mathbb{W}_1}

\nline

&=&a\left(\left(v,z\right),\Phi_1\right)-b\left(\Phi_2,\Phi_2\right)-a\left(\Phi_1,\left(v,z\right)\right).

\end{array}
\end{equation*}
Finally, taking the real parts of the above equation, then using \eqref{E:=(2.7)}, we get
\begin{equation}\label{E:=(3.1)}
\Re\left\{\left<\mathcal{A}\Phi,\Phi\right>_{\mathcal{H}}\right\}=-\kappa\int_0^{L_0}\left|\theta_x\right|^2 dx\leq 0.
\end{equation}
We next prove the maximality. For $f=\left(g,h\right)\in \mathcal{H}=\mathbb{W}_1\times {\mathbb{V}_2}$, such that $g=\left(g_1,g_2\right)$ and $h=\left(h_1,h_2,\zeta\right),$ we show the existence of $\Phi=\left(\Phi_1,\Phi_2\right)\in D\left(\mathcal{A}\right)$, such that $\Phi_1=\left(u,y\right)$ and $\Phi_2=\left(v,z,\theta\right)$, unique solution of the equation
\begin{equation}\label{E:=(3.2)}
\Phi-\mathcal{A}\Phi=F,
\end{equation}
that is
\begin{equation*}
\left\{
\begin{array}{lll}

\Phi_1-\left(v,z\right)=g,

\nline

\Phi_2+C^{-1}\left(B\Phi_2+A_1\Phi_1\right)=h.

\end{array}
\right.
\end{equation*}
Since the operator $C$  is an isomorphism of ${\mathbb{V}_2}$ onto $\mathbb{V}_2'$, then the above system is equivalent to
\begin{equation}\label{E:=(3.3)}
\left\{
\begin{array}{lll}

\Phi_1=\left(v,z\right)+g,

\nline

C\Phi_2+B\Phi_2+A_1\Phi_1=Ch.

\end{array}
\right.
\end{equation}
Inserting the first equation of  \eqref{E:=(3.3)} in the second equation, we obtain that
\begin{equation}\label{E:=(3.4)}
C\Phi_2+B\Phi_2+A_1\left(v,z\right)=Ch-A_1g.
\end{equation}
Since $h\in {\mathbb{V}_2}$, $g \in \mathbb{W}_1$, and the operator $C$ (resp. $A_0$) is an isomorphism of ${\mathbb{V}_2}$ onto $\mathbb{V}_2'$  (resp. $\mathbb{W}_1$ onto $\mathbb{W}_1'$),  then using \eqref{E:=(2.5)} and the definition of $A_1$ (see last equation of \eqref{E:=(2.8)}), we get
\begin{equation*}
R:=Ch-A_1g\in \mathbb{W}_1'\times L^2(0,L_0).
\end{equation*}
Using the above equation in \eqref{E:=(3.4)}, we get
\begin{equation}\label{E:=(3.5)}
C\Phi_2+B\Phi_2+A_1\left(v,z\right)=R\quad \text{in }\mathbb{W}_1'\times L^2(0,L_0).
\end{equation}
We define the operator $D\in \mathcal{L}\left({\mathbb{W}_2},\mathbb{W}_2'\right)$ by
\begin{equation*}
\begin{array}{lll}
\text{for }  {Z}=\left(Z_1,\theta\right),\ \hat{Z}=\left(Z_2, \hat{\theta}\right)\in {\mathbb{W}_2}, \text{ such that } Z_1=\left(v,z\right),\ \hat{Z}_2=\left(\hat{v},\hat{z}\right):

\nline

\left<D{Z},\hat{Z}\right>_{\mathbb{W}_2'\times {\mathbb{W}_2}}:=\left<CZ+BZ+A_1 Z_1,\hat{Z}\right>_{\mathbb{W}_2'\times {\mathbb{W}_2}}=c\left(Z,\hat{Z}\right)+b\left(Z,\hat{Z}\right)+a\left(Z_1,\hat{Z}_1\right).

\end{array}
\end{equation*}
From \eqref{E:=(2.7)} and \eqref{E:=(2.9)}, we have
\begin{equation*}
\Re\left\{\left<D{Z},\hat{Z}\right>_{\mathbb{W}_2'\times {\mathbb{W}_2}}\right\}=\|Z\|_{{\mathbb{V}_2}}^2+\kappa\int_0^{L_0}\left|\theta_x\right|^2 dx+\|Z_1\|_{{\mathbb{W}_1}}^2\geq \min\left(1+\kappa\right)\|Z\|_{{\mathbb{W}_2}}^2.
\end{equation*}
 So, by using Lax-Milgram lemma, for all $T\in \mathbb{W}_2'$, we get that $DZ=T$ has a unique solution $Z=\left(v,z,\theta\right)\in \mathbb{W}_2$. Consequently, since $R\in \mathbb{W}_1'\times L^2(0,L_0)\subset \mathbb{W}_1'\times H^{-1}(0,L_0)=\mathbb{W}_2'$, we get that \eqref{E:=(3.5)} has a unique solution $\Phi_2=\left(v,z,\theta\right)\in \mathbb{W}_2$. Next, we define $\Phi_1:=\left(v,z\right)+g$. Since $g,\ \left(v,z\right)  \in \mathbb{W}_1$, we get $\Phi_1 \in \mathbb{W}_1$. Consequently, $\left(\Phi_1,\Phi_2\right)\in \mathbb{W}_1\times \mathbb{W}_2$ is the unique solution of \eqref{E:=(3.3)}. In addition, since $h\in {\mathbb{V}_2}$, $\Phi_2 \in \mathbb{W}_2\subset \mathbb{V}_2$, and the operator $C$ is an isomorphism of ${\mathbb{V}_2}$ onto $\mathbb{V}_2'$, we get
\begin{equation*}
 B\Phi_2+A_1\Phi_1=Ch-C\Phi_2\in \mathbb{V}_2'.
\end{equation*}
Thus, \eqref{E:=(3.2)} has a unique solution $\Phi:=\left(\Phi_1,\Phi_2\right)\in D\left(\mathcal{A}\right),$ completing the proof of the proposition.
\end{proof}
\noindent	Thanks to  Lumer-Phillips theorem (see \cite{Liu-Zheng-1999, Pazy-1983}), we deduce that $\mathcal{A}$ generates a  $C_0$-semigroup of contractions $\left(e^{t\mathcal{A}}\right)_{t\geq 0}$ in $\mathcal{H}$ and therefore  Problem \eqref{E:=(2.14)} is well-posed. Then, we have the following result.
\begin{Theorem}\label{T:=3.2}
For any $\Phi_0\in\mathcal{H}$, the Problem \eqref{E:=(2.14)}  admits a unique weak solution
$\Phi:=e^{t\mathcal{A}}\Phi_0\in C\left(\mathbb{R}_{+};\mathcal{H}\right).$
Moreover, if $\Phi_0\in D\left(\mathcal{A}\right),$ then
$\Phi\in C\left(\mathbb{R}_{+};D\left(\mathcal{A}\right)\right) \cap C^1\left(\mathbb{R}_{+};\mathcal{H}\right).$\xqed{$\square$}
\end{Theorem}
\subsection{Strong stability of the system}\label{S:=3.2}
Our main result in this part is the following theorem.
\begin{Theorem}\label{T:=3.3}
The semigroup of contractions $\left(e^{t\mathcal{A}}\right)_{t\geq0}$ is strongly stable on ${\mathcal{H}}$ in the sense that
\begin{equation*}
\displaystyle{\lim_{t\to +\infty}\|e^{t{\mathcal{A}}}\Phi_0\|_{{\mathcal{H}}}=0},\quad \forall\;\Phi_0\in {\mathcal{H}}.
\end{equation*}
\end{Theorem}
\noindent For the proof of Theorem \ref{T:=3.3}: First we will prove that the
the operator ${\mathcal{A}}$  has a compact resolvent on the energy space $\mathcal{H}$. Then, we will establish that ${\mathcal{A}}$  has no eigenvalues on the imaginary axis. The proof for Theorem \ref{T:=3.3} relies on the subsequent lemmas.
\begin{Lemma}\label{L:=3.4}
 Let $\Phi=\left(\Phi_1,\Phi_2\right)\in D\left(\mathcal{A}\right)$, such that $\Phi_1=\left(u,y\right),$ and $\Phi_2=\left(v,z,\theta\right)$. Then, we have
\begin{eqnarray}
\left(v,z\right)\in \mathbb{W}_1,\label{E:=(3.6)}\nline
\theta\in H^2(0,L_0)\cap H^1_0(0,L_0),\label{E:=(3.7)}\nline
\left(u,y\right)\in \mathbb{W}_1\cap \left(H^3\left(0,L_0\right)\times H^3\left(L_0,L\right)\right)\label{E:=(3.8)},\nline
 \beta_1 u_{xx}(L_0)=\beta_2 y_{xx}(L_0).\label{E:=(3.9)}
\end{eqnarray}
In particular, the resolvent $\left(I-\mathcal{A}\right)^{-1}$ of $\mathcal{A}$ is compact on the energy space $\mathcal{H}$.
\end{Lemma}
\begin{proof} The proof is divided into 3 steps.\\[0.1in]
$\bullet$ \textbf{Step 1.} In this step, we write the variational problem and we prove \eqref{E:=(3.6)}. For this aim, let $f=\left(g,h\right)\in \mathcal{H}=\mathbb{W}_1\times {\mathbb{V}_2}$ and $\Phi=\left(\Phi_1,\Phi_2\right)\in D\left(\mathcal{A}\right)$, such that
\begin{equation}\label{E:=(3.10)}
\mathcal{A}\Phi= f,
\end{equation}
where $g=\left(g_1,g_2\right)$, $h=\left(h_1,h_2,h_3\right),$ $\Phi_1=\left(u,y\right),$ and $\Phi_2=\left(v,z,\theta\right)$. Equation \eqref{E:=(3.10)} is equivalent to
\begin{eqnarray*}
\left(v,z\right)=\left(g_1,g_2\right)\in \mathbb{W}_1,
\nline
B\Phi_2+A_1\Phi_1=-Ch\in \mathbb{V}_2' \subset \mathbb{W}_2'.
\end{eqnarray*}
From the first equation of the above system, we obtain \eqref{E:=(3.6)}. For all $Z=\left(\phi,\varphi,\psi\right)\in \mathbb{W}_2$, using the above equation, \eqref{E:=(2.5)}, and  \eqref{E:=(2.8)},  one gets
\begin{equation*}
\begin{array}{lll}

\left<B\Phi_2+A_1\Phi_1,Z\right>_{\mathbb{V}_2'\times {\mathbb{V}_2}}&=&-\left<C h,Z\right>_{\mathbb{V}_2'\times {\mathbb{V}_2}},

\nline

\left<B\Phi_2+A_1\Phi_1,Z\right>_{\mathbb{W}_2'\times {\mathbb{W}_2}}&=&-c\left(h,Z\right),

\nline

b\left(\Phi_2,Z\right)+a\left(\Phi_1,\left(\phi,\varphi\right)\right)&=&-c\left(h,Z\right).
\end{array}
\end{equation*}
Using  \eqref{E:=(2.6)} in the above equation, we obtain that for all $\left(\phi,\varphi,\psi\right)\in \mathbb{W}_2$:
\begin{equation}\label{E:=(3.11)}
\begin{array}{lll}

\beta_1\int_0^{L_0}u_{xx}\overline{\phi_{xx}}\, dx+\beta_2\int_{L_0}^Ly_{xx}\overline{\varphi_{xx}}\, dx
-\gamma\int_0^{L_0}\theta_{x}\overline{{\phi}_x}\, dx+\kappa \int_0^{L_0}\theta_{x}\overline{\psi_x}\, dx=-\rho_0\int_0^{L_0} h_3 \overline{\psi}\, dx

\nline

-\gamma\int_0^{L_0}(g_1)_{x}\overline{{\psi}_x}\, dx-\int_0^{L_0}\left(\rho_1h_1\overline{\phi}+\alpha_1\, (h_1)_x\overline{\phi_x}\right)dx-\int_{L_0}^L\left(\rho_2h_2\overline{\varphi}+\alpha_2\, (h_2)_x\overline{\varphi_x}\right)dx.

\end{array}
\end{equation}
$\bullet$ \textbf{Step 2.} In this step, we prove \eqref{E:=(3.7)} and
\begin{equation}\label{E:=(3.12)}
\theta(x)=-\frac{\gamma}{\kappa} g_1(x)+\frac{ \rho_0}{\kappa}\int_0^{x}\int_0^{x_1} h_3(x_2)\, dx_2\, dx_1,\quad \forall x\in (0,L_0).
\end{equation}
For this aim,  setting $\phi=0,\ \varphi=0,$ and $\psi\in H^1_0\left(0,L_0\right)$ in \eqref{E:=(3.11)},  we obtain
\begin{equation}\label{E:=(3.13)}
\kappa \int_0^{L_0}\theta_{x}\overline{\psi}_x\, dx=-\int_0^{L_0}\left(\gamma\, (g_1)_{x} \overline{\psi_x}+\rho_0 h_3 \overline{\psi}\right)   dx\quad \forall \psi\in H^1_0\left(0,L_0\right).
\end{equation}
The left hand side of \eqref{E:=(3.13)} is a bilinear continuous coercive form on $H^1_0\left(0,L_0\right)\times H^1_0\left(0,L_0\right)$,  while  the right hand side is a linear continuous form on $H^1_0\left(0,L_0\right)$. Then, using Lax-Milgram lemma, we deduce that there exists unique  $\theta\in H^1_0\left(0,L_0\right)$  solution of the variational problem \eqref{E:=(3.13)}. Applying classical regularity arguments, we infer that $\theta\in H^1_0\left(0,L_0\right)\cap H^2(0,L_0)$, hence we get \eqref{E:=(3.7)}. Consequently, setting $\psi\in C^{\infty}_c\left(0,L_0\right)\subset H^1_0\left(0,L_0\right)$ in \eqref{E:=(3.13)}, then using integration by parts,  we obtain
\begin{equation*}
 \int_0^{L_0}\left(-\kappa\theta_{xx}-\gamma\, (g_1)_{xx}+\rho_0 h_3 \right) \overline{\psi}\,  dx=0,\quad \forall\; \psi\in C^{\infty}_c\left(0,L_0\right).
\end{equation*}
Thus, by applying Corollary 4.24 in \cite{Brezis-2010}, we get
\begin{equation*}
\kappa \theta_{xx}=-\gamma\, (g_1)_{xx}+\rho_0 h_3\in L^2(0,L_0),\quad \text{a.e. } x\in (0,L_0).
\end{equation*}
 Solving the above differential equation (taking into consideration that $(g_1,g_2)\in \mathbb{W}_1$; i.e., $g_1(0)=(g_1)_x(0)=0$), we obtain \eqref{E:=(3.12)}.\\[0.1in]
$\bullet$ \textbf{Step 3.} In this step, we prove \eqref{E:=(3.8)} and \eqref{E:=(3.9)}. For this aim, setting $\left(\phi, \varphi\right)\in  \mathbb{W}_1,$ and $\psi=0$ in \eqref{E:=(3.11)},  then using \eqref{E:=(3.12)}, we obtain
\begin{equation}\label{E:=(3.14)}
\begin{array}{lll}

\beta_1\int_0^{L_0}u_{xx}\overline{\phi_{xx}}\, dx+\beta_2\int_{L_0}^Ly_{xx}\overline{\varphi_{xx}}\, dx=

-\kappa^{-1}\int_0^{L_0}\left(\gamma\, (g_1)_x-\rho_0\int_0^x h(x_2)\, dx_2\right)\overline{{\phi}_x}\, dx

\nline

-\int_0^{L_0}\left(\rho_1h_1\overline{\phi}+\alpha_1\, (h_1)_x\overline{\phi_x}\right)dx

-\int_{L_0}^L\left(\rho_2h_2\overline{\varphi}+\alpha_2\, (h_2)_x\overline{\varphi_x}\right)dx,\ \ \forall \left(\phi,\varphi\right)\in \mathbb{W}_1.

\end{array}
\end{equation}
The left hand side of \eqref{E:=(3.14)} is a bilinear continuous coercive form on $\mathbb{W}_1\times \mathbb{W}_1$,  while  the right hand side is a linear continuous form on $\mathbb{W}_1$. Then, using Lax-Milgram lemma, we deduce that there exists unique  $\left(u,y\right)\in \mathbb{W}_1$  solution of the variational problem \eqref{E:=(3.14)}. Now, fix $\left(\nu,\mu\right)\in C^{\infty}_c\left(0,L_0\right)\times C^{\infty}_c\left(L_0,L\right)$ such that
\begin{equation*}
\int_0^{L_0}\nu\, dx=\int_0^{L_0}\mu\, dx=1.
\end{equation*}
For any function $\left(\hat{\nu},\hat{\mu}\right)\in \mathbb{W}_1,$ we define
\begin{equation}\label{E:=(3.15)}
\left\{
\begin{array}{lll}

\phi(x)=\int_0^x\left(\hat{\nu}(x_1)-\left(\int_0^{L_0}\hat{\nu}(x_2)\, dx_2\right)\nu(x_1)  \right)  dx_1,\quad \forall x\in (0,L_0),

\nline

\varphi(x)=-\int_{x}^L\left(\hat{\mu}(x_1)-\left(\int_{L_0}^L\hat{\mu}(x_2)\, dx_2\right)\mu(x_1)  \right)  dx_1,\quad \forall (L_0,L).

\end{array}
\right.
\end{equation}
Indeed, the function $\left(\phi,\varphi\right)\in C^{1}\left(0,L_0\right)\times C^{1}\left(L_0,L\right)$  and
\begin{equation*}
\phi(0)=\phi(L_0)=\varphi(L_0)=\varphi(L)=\phi_x(0)=\phi_x(L)=0\quad \text{and}\quad \varphi_x(L_0)=\varphi_x(L_0)=\hat{\nu}(L_0)=\hat{\mu}(L_0).
\end{equation*}
Thus $\left(\phi,\varphi\right)\in  \mathbb{W}_1$, and consequently, by substituting \eqref{E:=(3.15)} in \eqref{E:=(3.14)},  we derive
\begin{equation}\label{E:=(3.16)}
\begin{array}{lll}

\beta_1\int_0^{L_0}u_{xx}\overline{\hat{\nu}}_x\, dx-\beta_1 \int_0^{L_0}\overline{\hat{\nu}}\,  dx\, \int_0^{L_0}u_{xx} \overline{\nu}_x\,  dx+\beta_2\int_{L_0}^L y_{xx}\overline{\hat{\mu}}_x\, dx-\beta_2 \int_{L_0}^L\overline{\hat{\mu}}\,  dx\, \int_{L_0}^Ly_{xx} \overline{\mu}_x\,  dx=

\nline

-\kappa^{-1}\int_0^{L_0}\left(\gamma\, (g_1)_x-\rho_0\int_0^x h(x_2)\, dx_2\right)\left(\overline{\hat{\nu}}-\left(\int_0^{L_0}\overline{\hat{\nu}}(x_2)\, dx_2\right)\overline{\nu}\right) dx

\nline

-\rho_1\int_0^{L_0}h_1 \left[\int_0^x\left(\overline{\hat{\nu}}(x_1)-\left(\int_0^{L_0}\overline{\hat{\nu}}(x_2)\, dx_2\right)\overline{\nu}(x_1)  \right)  dx_1\right]  dx

 \nline

-\alpha_1\int_0^{L_0} (h_1)_x \left(\overline{\hat{\nu}}-\left(\int_0^{L_0}\overline{\hat{\nu}}(x_2)\, dx_2\right)\overline{\nu}\right) dx

\nline

+\rho_2\int_{L_0}^Lh_2\left[\int_{x}^L\left(\overline{\hat{\mu}}(x_1)-\left(\int_{L_0}^L\overline{\hat{\mu}}(x_2)\, dx_2\right)\overline{\mu}(x_1)  \right)  dx_1\right] dx

\nline

-\alpha_2\int_{L_0}^L (h_2)_x\left(\overline{\hat{\mu}}-\left(\int_{L_0}^L\overline{\hat{\mu}}(x_2)\, dx_2\right)\overline{\,u}\right) dx.

\end{array}
\end{equation}
 We have
\begin{equation*}
\begin{array}{lll}
-\rho_1\int_0^{L_0}h_1 \left[\int_0^x\left(\overline{\hat{\nu}}(x_1)-\left(\int_0^{L_0}\overline{\hat{\nu}}(x_2)\, dx_2\right)\overline{\nu}(x_1)  \right)  dx_1\right]  dx =

\nline

-\rho_1\int_0^{L_0}\left(\int_0^x h_1(x_1)\, dx_1\right)_x \left(\int_0^x\overline{\hat{\nu}}(x_1)  dx_1\right)  dx+\int_0^{L_0} \left[\int_0^{L_0}h_1(x_1) \left( \int_0^{x_1}\overline{\nu}(x_2)\,  dx_2\right)  dx_1\right] \overline{\hat{\nu}}\, dx.

\end{array}
\end{equation*}
In the above equation, for the first term, using integration by parts, we get
\begin{equation}\label{E:=(3.17)}
\begin{array}{lll}
-\rho_1\int_0^{L_0}h_1 \left[\int_0^x\left(\overline{\hat{\nu}}(x_1)-\left(\int_0^{L_0}\overline{\hat{\nu}}(x_2)\, dx_2\right)\overline{\nu}(x_1)  \right)  dx_1\right]  dx =

\nline

\rho_1\int_0^{L_0}\left[\int_0^x h_1(x_1)\, dx_1-\int_0^{L_0} h_1(x_1)\, dx_1+\int_0^{L_0}h_1(x_1) \left( \int_0^{x_1}\overline{\nu}(x_2)\,  dx_2\right)  dx_1\right] \overline{\hat{\nu}}\,  dx.

\end{array}
\end{equation}
By the same way, using integration by parts, we get
\begin{equation}\label{E:=(3.18)}
\begin{array}{lll}
\rho_2\int_{L_0}^Lh_2\left[\int_{x}^L\left(\overline{\hat{\mu}}(x_1)-\left(\int_{L_0}^L\overline{\hat{\mu}}(x_2)\, dx_2\right)\overline{\mu}(x_1)  \right)  dx_1\right] dx =

\nline

-\rho_2\int_{L_0}^L\left[\int_x^L  h_2(x_1)\, dx_1-\int_{L_0}^L  h_2(x_1)\, dx_1+\int_{L_0}^{L} h_2(x_1) \left(\int_{x_1}^L\overline{\mu}(x_2)\,  dx_2\right) dx_1\right] \overline{\hat{\mu}}\,  dx.

\end{array}
\end{equation}
Next, we have
\begin{equation}\label{E:=(3.19)}
\begin{array}{lll}

-\kappa^{-1}\int_0^{L_0}\left(\gamma\, (g_1)_x-\rho_0\int_0^x h(x_2)\, dx_2\right)\left(\overline{\hat{\nu}}-\left(\int_0^{L_0}\overline{\hat{\nu}}(x_2)\, dx_2\right)\overline{\nu}\right) dx

\nline

=-\kappa^{-1}\int_0^{L_0}\left[\gamma\, (g_1)_x-\rho_0\int_0^x h(x_2)\, dx_2-\int_0^{L_0}\left(\gamma\,(g_1)_x-\rho_0\int_0^x h(x_2)\, dx_2\right) \overline{\nu}\,  dx\right]\overline{\hat{\nu}}\, dx,

\end{array}
\end{equation}
\begin{equation}\label{E:=(3.20)}
-\alpha_1\int_0^{L_0} (h_1)_x \left(\overline{\hat{\nu}}-\left(\int_0^{L_0}\overline{\hat{\nu}}(x_2)\, dx_2\right)\overline{\nu}\right) dx=-\alpha_1\int_0^{L_0}\left[(h_1)_x-\int_0^{L_0} (h_1)_x \overline{\nu}\, dx\right]\overline{\hat{\nu}}\, dx,
\end{equation}
and
\begin{equation}\label{E:=(3.21)}
-\alpha_2\int_{L_0}^L (h_2)_x\left(\overline{\hat{\mu}}-\left(\int_{L_0}^L\overline{\hat{\mu}}(x_2)\, dx_2\right)\overline{\,u}\right) dx=-\alpha_2\int_{L_0}^L\left[(h_2)_x-\int_{L_0}^L (h_2)_x \overline{\mu}\, dx\right]\overline{\hat{\mu}}\, dx.
\end{equation}
Replacing \eqref{E:=(3.17)}-\eqref{E:=(3.21)} in \eqref{E:=(3.16)}, we obtain
\begin{equation}\label{E:=(3.22)}
\beta_1\int_0^{L_0}u_{xx}\overline{\hat{\nu}}_x\, dx+\beta_2\int_{L_0}^L y_{xx}\overline{\hat{\mu}}_x\, dx=\int_0^{L_0}\chi_1 \overline{\hat{\nu}}\, dx+\int_{L_0}^{L}\chi_2 \overline{\hat{\nu}}\, dx,\quad \forall\; \left(\hat{\nu},\hat{\mu}\right)\in \mathbb{W}_1,
\end{equation}
where
\begin{equation*}
\begin{array}{llll}

\chi_1(x)=-\frac{\gamma}{\kappa}(g_1)_x(x)-\alpha_1\, (h_1)_x(x)+\frac{\rho_0}{\kappa}\int_0^x h(x_2)\, dx_2+\rho_1\int_0^x h_1(x_1)\, dx_1

\nline

+ \beta_1\int_0^{L_0}u_{xx} \overline{\nu}_x\,  dx

-\rho_1\int_0^{L_0} h_1(x_1)\, dx_1+\rho_1\int_0^{L_0}h_1(x_1) \left( \int_0^{x_1}\overline{\nu}(x_2)\,  dx_2\right)  dx_1

\nline

+\frac{1}{\kappa}\int_0^{L_0}\left(\gamma\, (g_1)_x-\rho_0\int_0^x h(x_2)\, dx_2\right) \overline{\nu}\,  dx+\alpha_1 \int_0^{L_0} (h_1)_x \overline{\nu}\, dx\in L^2(0,L_0)

\end{array}
\end{equation*}
and
\begin{equation*}
\begin{array}{llll}

\chi_2(x)=-\alpha_2\, (h_2)_x(x)-\rho_2\int_x^L  h_2(x_1)\, dx_1 +\beta_2\int_{L_0}^Ly_{xx} \overline{\mu}_x\,  dx+\rho_2\int_{L_0}^L  h_2(x_1)\, dx_1

\nline

 -\rho_2\int_{L_0}^{L} h_2(x_1) \left(\int_{x_1}^L\overline{\mu}(x_2)\,  dx_2\right) dx_1+\alpha_2\int_{L_0}^L (h_2)_x \overline{\mu}\, dx\in L^2(L_0,L).

\end{array}
\end{equation*}
Taking $\left(\hat{\nu},\hat{\mu}\right)\in  C^{1}_c\left(0,L_0\right)\times C^{1}_c\left(L_0,L\right)\subset  \mathbb{W}_1$ in \eqref{E:=(3.22)}, we get that
\begin{equation}\label{E:=(3.23)}
\begin{array}{lll}

\forall \left(\hat{\nu},\hat{\mu}\right)\in C^{1}_c\left(0,L_0\right)\times C^{1}_c\left(L_0,L\right):

\nline

\beta_1\int_0^{L_0}u_{xx}\overline{\hat{\nu}}_x\, dx+\beta_2\int_{L_0}^L y_{xx}\overline{\hat{\mu}}_x\, dx=\int_0^{L_0}\chi_1 \overline{\hat{\nu}}\, dx+\int_{L_0}^{L}\chi_2 \overline{\hat{\nu}}\, dx,

\end{array}
\end{equation}
thus, by using the definition of $H^1$ (see page 202 in \cite{Brezis-2010}), we get $\left(u_{xx},y_{xx}\right)\in H^1(0,L_0)\times H^1(L_0, L)$, and consequently \eqref{E:=(3.8)} holds true. Back to \eqref{E:=(3.23)}, using integration by parts in the left hand side, one derives
\begin{equation*}
\begin{array}{lll}
\int_0^{L_0}\left(-\beta_1 u_{xxx}-\chi_1\right)\overline{\hat{\nu}}\, dx+\int_{L_0}^L \left(-\beta_2 y_{xxx}-\chi_2\right)\overline{\hat{\mu}}\, dx=0,\quad \forall\; \left(\hat{\nu},\hat{\mu}\right)\in C^{1}_c\left(0,L_0\right)\times C^{1}_c\left(L_0,L\right),
\end{array}
\end{equation*}
consequently, by applying Corollary 4.24 in \cite{Brezis-2010}, we obtain
\begin{equation}\label{E:=(3.24)}
-\beta_1 u_{xxx}=\chi_1\ \ \text{for a.e. } x\in (0,L_0)\quad \text{and}\quad -\beta_2 y_{xxx}=\chi_2 \ \ \text{for a.e. } x\in (L_0,L).
\end{equation}
Finally, using integration by parts in the left hand side of \eqref{E:=(3.22)}, then using  \eqref{E:=(3.24)} and  taking $\hat{\mu}(L_0)=\hat{\nu}(L_0)=1$, it holds that
\begin{equation*}
    \beta_1 u_{xx}(L_0)-\beta_2 y_{xx}(L_0)=0,
\end{equation*}
thus, we obtain  \eqref{E:=(3.9)}.\\[0.1in]
$\bullet$ \textbf{Step 4.}  In this step, we prove that the resolvent $\left(I-\mathcal{A}\right)^{-1}$ of $\mathcal{A}$ is compact on the energy space $\mathcal{H}$. For this aim, let $f\in \mathcal{H}$ and $\Phi\in D\left(\mathcal{A}\right)$, such that $\left(I-\mathcal{A}\right)\Phi= f.$ Since $\mathcal{A}$ is monotone, it follows that
\begin{equation*}
\left\|f\right\|_{\mathcal{H}}^2\geq \left\|\Phi\right\|^2_{\mathcal{H}}+\left\|\mathcal{A}\Phi\right\|^2_{\mathcal{H}}.
\end{equation*}
The result follows from the above inequality and \eqref{E:=(3.6)}-\eqref{E:=(3.8)}. This completes the proof of the lemma.
\end{proof}
\begin{Lemma}\label{L:=3.5}
For all $\lambda\in \mathbb{R}$, we have
\begin{equation*}
\ker\left(i\lambda I-\mathcal{A}\right)=\{0\},
\end{equation*}
 where $\ker\left(i\lambda I-\mathcal{A}\right)$ denotes the Kernel of $i\lambda I-\mathcal{A}$.
\end{Lemma}
\begin{proof} Let $\lambda\in\mathbb{R}$, such that $i\lambda$  is an eigenvalue of the operator $\mathcal{A}$ and $\Phi=\left(\Phi_1,\Phi_2\right)\in D\left(\mathcal{A}\right)$  a corresponding eigenvector, where $\Phi_1=\left(u,y\right)$ and $\Phi_2=\left(v,z,\theta\right)$ . Therefore, we have
\begin{equation}\label{E:=(3.25)}
\mathcal{A} \Phi=i\lambda \Phi.
\end{equation}
Similar to \eqref{E:=(3.1)}, we get
\begin{equation*}
0=\Re\left<i\lambda \Phi,\Phi\right>_{\mathcal{H}}=\Re\left<\mathcal{A} \Phi,\Phi\right>_{\mathcal{H}}=-\kappa\int_0^{L_0}\left|\theta_x\right|^2 dx.
\end{equation*}
Consequently, we deduce that
\begin{equation*}
\theta_x=0,\quad \text{a.e. } x\in (0,L_0).
\end{equation*}
Since $\theta\in H^1_0(0,L_0)$ (i.e., $\theta(0)=\theta(L_0)=0$), we get
\begin{equation}\label{E:=(3.26)}
\theta=0,\quad \text{a.e. } x\in (0,L_0).
\end{equation}
Next, writing \eqref{E:=(3.25)} in a detailed form gives
\begin{equation*}
\left\{
\begin{array}{lll}

(v,z)= \left(i\lambda\, u,i\lambda\, y\right),

\nline

C^{-1}\left(B\Phi_2+A_1\Phi_1\right)+i\lambda \Phi_2=0.

\end{array}
\right.
\end{equation*}
Since the operator $C$  is an isomorphism of ${\mathbb{V}_2}$ onto $\mathbb{V}_2'$, then the above system is equivalent to
\begin{equation}\label{E:=(3.27)}
\left\{
\begin{array}{lll}

(v,z)= \left(i\lambda\, u,i\lambda\, y\right),

\nline

B\Phi_2+A_1\Phi_1+i\lambda \, C\Phi_2=0.

\end{array}
\right.
\end{equation}
Inserting the first equation of  \eqref{E:=(3.27)} in the second one, then using \eqref{E:=(3.26)}, we obtain that
\begin{equation*}
A_1\left(u,y\right)+B \left(i\lambda\, u,i\lambda\, y,0\right)-\lambda^2 \, C\left( u, y,0\right)=0,\quad \text{in } \mathbb{V}_2'\subset \subset \mathbb{W}_2'.
\end{equation*}
For all $Z=\left(\phi,\varphi,\psi\right)\in \mathbb{W}_2$, using the above equation, \eqref{E:=(2.5)}, and  \eqref{E:=(2.8)},  we get
\begin{equation*}
\begin{array}{ccc}

0&=&\left<A_1\left(u,y\right)+B\left(i\lambda\, u,i\lambda\, y,0\right)-\lambda^2 \, C\left( u, y,0\right),\left(\phi,\varphi,\psi\right) \right>_{\mathbb{V}_2'\times {\mathbb{V}_2}}\nline

&=&\left<A_1\left(u,y\right)+B\left(i\lambda\, u,i\lambda\, y,0\right),\left(\phi,\varphi,\psi\right)\right>_{\mathbb{W}_2'\times {\mathbb{W}_2}}-\lambda^2 \left<C\left( u, y,0\right),\left(\phi,\varphi,\psi\right)\right>_{\mathbb{V}_2'\times {\mathbb{V}_2}}\nline

&=&a\left(\left( u, y\right),\left(\phi,\varphi\right)\right)+i\lambda b\left(\left( u, y,0\right),\left(\phi,\varphi,\psi\right)\right)-\lambda^2 c\left(\left( u, y,0\right),\left(\phi,\varphi,\psi\right)\right).

\end{array}
\end{equation*}
Using  \eqref{E:=(2.6)} in the above equation, we obtain
\begin{equation}\label{E:=(3.28)}
\begin{array}{lll}

\beta_1\int_0^{L_0}u_{xx}\overline{\phi_{xx}}\, dx+\beta_2\int_{L_0}^Ly_{xx}\overline{\varphi_{xx}}\, dx

+\gamma\int_0^{L_0} u_{x}\overline{\psi_x} dx=

\nline

\lambda^2\int_0^{L_0}\left(\rho_1u\overline{\phi}+\alpha_1\, u_x\overline{\phi_x}\right)dx

+\lambda^2\int_{L_0}^L\left(\rho_2y\overline{\varphi}+\alpha_2\, y_x\overline{\varphi_x}\right)dx,\ \ \forall \left(\phi,\varphi,\psi\right)\in \mathbb{W}_2.

\end{array}
\end{equation}
Now, setting $\phi=0,\ \varphi=0,$ and $\psi\in C^{\infty}_c\left(0,L_0\right)\subset H^1_0\left(0,L_0\right)$ in \eqref{E:=(3.28)},  we get
\begin{equation*}
\gamma\int_0^{L_0} u_{x}\overline{\psi_x} dx=0,\quad \forall \psi\in C^{\infty}_c\left(0,L_0\right).
\end{equation*}
In the above equation, using integration by parts, we see that
\begin{equation*}
\int_0^{L_0} u_{xx}\overline{\psi} dx=0,\quad \forall \psi\in C^{\infty}_c\left(0,L_0\right).
\end{equation*}
Therefore,  by applying Corollary 4.24 in \cite{Brezis-2010}, we have
\begin{equation*}
u_{xx}=0,\quad \text{a.e. } x\in (0,L_0).
\end{equation*}
Since $u\in H^2(0,L_0)$ and $u(0)=u_x(0)=0$, one derives
\begin{equation}\label{E:=(3.29)}
u=0,\quad \text{a.e. } x\in (0,L_0).
\end{equation}
Since $(u,y)\in \mathbb{W}_1$, then from \eqref{E:=(3.29)} and by the help of Lemma \ref{L:=3.4}, we obtain
\begin{equation}\label{E:=(3.30)}
y\in H^3(L_0,L)\quad \text{and}\quad y(L)=y_x(L)=y(L_0)=y_x(L_0)=y_{xx}(L_0)=0.
\end{equation}
Next, setting $\left(\phi, \varphi\right)\in   \mathbb{W}_1$ and $\psi=0$ in \eqref{E:=(3.28)},  then using \eqref{E:=(3.29)}, one has
\begin{equation*}
\beta_2\int_{L_0}^Ly_{xx}\overline{\varphi_{xx}}\, dx=\lambda^2\int_{L_0}^L\left(\rho_2y\overline{\varphi}+\alpha_2\, y_x\overline{\varphi_x}\right)dx,\quad \forall\, \left(\phi, \varphi\right)\in   \mathbb{W}_1.
\end{equation*}
Using integration by parts, after that using \eqref{E:=(3.30)} and the fact that $\varphi_x(L_0)=0$, we get
\begin{equation}\label{E:=(3.31)}
-\beta_2\int_{L_0}^Ly_{xxx}\overline{\varphi_{x}}\, dx=\lambda^2\int_{L_0}^L\left(\rho_2y-\alpha_2\, y_{xx}\right)\overline{\varphi}\, dx,\quad \forall\, \left(\phi, \varphi\right)\in   \mathbb{W}_1 .
\end{equation}
Taking $\left(\phi,\varphi\right)\in  C^{1}_c\left(0,L_0\right)\times C^{1}_c\left(L_0,L\right)\subset   \mathbb{W}_1$ in \eqref{E:=(3.22)}, we find that
\begin{equation}\label{E:=(3.32)}
-\beta_2\int_{L_0}^Ly_{xxx}\overline{\varphi_{x}}\, dx=\lambda^2\int_{L_0}^L\left(\rho_2y-\alpha_2\, y_{xx}\right)\overline{\varphi}\, dx,\quad  \forall\, \varphi\in    C^{1}_c\left(L_0,L\right),
\end{equation}
thus, by using the definition of $H^1$ (see page 202 in \cite{Brezis-2010}), we get $y_{xxx}\in H^1(L_0,L)$. Again, using integration by  parts in the right hand side of \eqref{E:=(3.32)}, we infer
\begin{equation*}
\int_{L_0}^L \left[\beta_2\, y_{xxxx}-\lambda^2\left(\rho_2y-\alpha_2\, y_{xx}\right)\right]\overline{\varphi}\, dx=0,\quad \forall\; \varphi\in C^{1}_c\left(L_0,L\right).
\end{equation*}
Consequently, by applying Corollary 4.24 in \cite{Brezis-2010},  then using  \eqref{E:=(3.30)}, we obtain
\begin{equation}\label{E:=(3.33)}
\beta_2 y_{xxxx}+\alpha_2\, \lambda^2\, y_{xx}-\rho_2\,\lambda^2\, y=0,\quad x\in (L_0,L).
\end{equation}
Using integration by parts in the left hand side of \eqref{E:=(3.31)},  after that using \eqref{E:=(3.33)} and the fact that $\phi(L)=0$, then   taking $\phi(L_0)=\varphi(L_0)=1$, one has
\begin{equation}\label{E:=(3.34)}
y_{xxx}(L_0)=0.
\end{equation}
Therefore, from \eqref{E:=(3.30)}, \eqref{E:=(3.33)} and \eqref{E:=(3.34)}, we get
\begin{eqnarray}
\beta_2 y_{xxxx}+\alpha_2\, \lambda^2\, y_{xx}-\rho_2\,\lambda^2\, y=0,\quad x\in (L_0,L),\label{E:=(3.35)}
\nline
y(L)=y_x(L)=y(L_0)=y_x(L_0)=y_{xx}(L_0)=y_{xxx}(L_0)=0.\label{E:=(3.36)}
\end{eqnarray}
Multiplying \eqref{E:=(3.35)} by $2\left(x-L\right)\overline{y_{xxx}}$, integrating by parts over $(L_0,L)$, then taking the real parts, we find
\begin{equation*}
\begin{array}{lll}

\beta_2 \int_{L_0}^L \left(x-L\right) \left(\left|y_{xxx}\right|^2\right)_x  dx+\alpha_2\lambda^2 \int_{L_0}^L \left(x-L\right) \left(\left|y_{xx}\right|^2\right)_x  dx-\rho_2\, \lambda^2 \int_{L_0}^L \left(x-L\right) \left(\left|y_{x}\right|^2\right)_x  dx

\nline

-4\rho_2\, \lambda^2 \int_{L_0}^L \left|y_x\right|^2  dx

-2\rho_2\,\lambda^2 \Re\left[\left(x-L\right) y \overline{ y_{xx} }-\left(\left(x-L\right) y\right)_x \overline{y_{x}}\right]_{L_0}^L=0.

\end{array}
\end{equation*}
Using integration by parts and the boundary conditions of \eqref{E:=(3.36)}, we arrive at
\begin{equation*}
-\beta_2 \int_{L_0}^L  \left|y_{xxx}\right|^2  dx-\alpha_2\lambda^2 \int_{L_0}^L \left|y_{xx}\right|^2  dx-3 \rho_2\,\lambda^2 \int_{L_0}^L \left|y_{x}\right|^2  dx=0.
\end{equation*}
Consequently, from the above equation and the boundary conditions of \eqref{E:=(3.36)}, we obtain
\begin{equation*}
y=0.
\end{equation*}
Finally, from the above equation, first equation of \eqref{E:=(3.27)}, \eqref{E:=(3.26)}, and \eqref{E:=(3.29)}, we get $\Phi=0.$  The proof is thus complete.
\end{proof}
\noindent \textbf{Proof of Theorem \ref{T:=3.3}.} From Lemma \ref{L:=3.4}, we have that the operator $\mathcal{A}$ has a compact resolvent. In addition, from Lemma \eqref{L:=3.5}, we get that the operator $\mathcal{A}$ has no pure imaginary eigenvalues. Thus, we get the conclusion by applying  Arendt and Batty theorem (see Theorem \ref{T:=A.2} and Corollary \ref{C:=A.3}).
\xqed{$\square$}
\section{Polynomial Stability}\label{S:=4}
\noindent In this section, we will prove the polynomial stability of System \eqref{E:=(1.1)}-\eqref{E:=(1.12)}. Our main results in this part are the following theorems.
\begin{Theorem}\label{T:=4.1}  If
\begin{equation}\label{E:=(4.1)}
\rho_1\geq \rho_2\quad\text{and}\quad \alpha_1\geq \alpha_2,
\end{equation}
then for all initial data $\Phi_0 \in D(\mathcal{A}),$ there exists a constant $C>0$ independent of $\Phi_0$ such that the energy of System \eqref{E:=(1.1)}-\eqref{E:=(1.12)} satisfies the following estimation
 \begin{equation*}
 E_1(t) \leq  \frac{C}{t^2} \left\|\Phi_0\right\|^2_{D(\mathcal{A})}, \quad \forall t>0.
 \end{equation*}
 \end{Theorem}
\begin{Theorem}\label{T:=4.2}
If
\begin{equation}\label{E:=(4.2)}
\rho_1< \rho_2\quad\text{or}\quad\alpha_1< \alpha_2,
\end{equation}
then for all initial data $\Phi_0 \in D(\mathcal{A}),$ there exists a constant $C>0$ independent of $\Phi_0$ such that the energy of System \eqref{E:=(1.1)}-\eqref{E:=(1.12)} satisfies the following estimation
 \begin{equation*}
 E_1(t) \leq  \frac{C}{t} \left\|\Phi_0\right\|^2_{D(\mathcal{A})}, \quad \forall t>0.
 \end{equation*}
 \end{Theorem}
\noindent From Lemma \ref{L:=3.4}   and Lemma \ref{L:=3.5}, we have seen that $i \mathbb{R}\subset \rho(\mathcal{A}),$  then for the proof of Theorems \ref{T:=4.1} and \ref{T:=4.2}, according to Borichev and Tomilov  \cite{Borichev-Tomilov-2010} (see Theorem \ref{T:=A.4}), we need to prove that
\begin{equation}\label{E:=(4.3)}
\sup_{\lambda\in\mathbb{R}}\frac{1}{|\lambda|^{\ell}}\left\|\left(i\lambda I-\mathcal{A}\right)^{-1}\right\|_{\mathcal{L}\left(\mathcal{H}\right)}<\infty,
\end{equation}
 where $\ell=1$ (resp. $\ell=2$) if condition \eqref{E:=(4.1)} (resp. condition \eqref{E:=(4.2)})  holds. We will argue by contradiction. We suppose that there exists
 \begin{equation*}
 \left\{(\lambda_n,\Phi^n:=\left(\Phi_1^n,\Phi_2^n\right))\right\}_{n\geq 1}\subset \mathbb{R}_{+}^*\times D\left(\mathcal{A}\right),
 \end{equation*}
 such that
\begin{equation}\label{E:=(4.4)}
\lambda_n \to  + \infty, \quad \left\|\Phi^n\right\|_{{\mathcal{H}}} = 1,
\end{equation}
and there exists a sequence ${F}^n:=(g^{n},h^{n}) \in \mathcal{H}$, such that
\begin{equation}\label{E:=(4.5)}
\lambda_n^{\ell} (i \lambda_n I - \mathcal{A})\Phi_n = F^n \to  0 \ \textrm{in} \ {\mathcal{H}},
\end{equation}
where $\Phi_1^n=\left(u^n,y^n\right),\ \Phi_2^n=\left(v^n,z^n,\theta^n\right),\ g^n=\left(g_1^n,g_2^n\right),$ and $h^n=\left(h_1^n,h_2^n,h_3^n\right).$  We will check condition \eqref{E:=(4.3)}  by finding a contradiction with $\left\|\Phi^n\right\|_{{\mathcal{H}}} = 1$ such as $\left\| \Phi^n\right\|_{\mathcal{H}} =o(1).$
By detailing Equation \eqref{E:=(4.5)}, we get the following system
\begin{equation}\label{E:=(4.6)}
\left\{
\begin{array}{lll}

i\lambda_n\, \Phi_1^n-\left(v^n,z^n\right)-\lambda_n^{-\ell}\, g^n=0,& \text{in } \mathbb{W}_1

\nline

i\lambda_n\, \Phi_2^n+C^{-1}\left(B\Phi_2^n+A_1\Phi_1^n\right)-\lambda_n^{-\ell}\, h^n=0,& \text{in } \mathbb{V}_2.

\end{array}
\right.
\end{equation}
The proof of Theorems \ref{T:=4.1}  and \ref{T:=4.2}  is divided into several lemmas.
\begin{Lemma}\label{L:=4.3}
For all $\ell\geq0$, the solution $\left(\Phi_1^n,\Phi_2^n\right)\in D(\mathcal{A})$ of System \eqref{E:=(4.6)} satisfies the following asymptotic behavior  estimations
\begin{eqnarray}
\left(\beta_1\, u_{xxx}^n+\alpha_1\,\lambda^{-\ell}_n   (h_1^n)_x+\alpha_1\,\lambda^2_n  u_{x}^n+\gamma  \theta_{x}^n+i\alpha_1\,\lambda^{1-\ell}_n   (g_1^n)_x\right)_x-\rho_1\,\lambda^2_n u^n=\rho_1\,\lambda^{1-\ell}_n\left(\lambda_n^{-1} h_1^n+ig_1^n\right),\label{E:=(4.7)}\nline
\left(\beta_2\, y_{xxx}^n+\alpha_2\,\lambda^{-\ell}_n   (h_2^n)_x+\alpha_2\,\lambda^2_n  y_{x}^n+i\alpha_2\,\lambda^{1-\ell}_n   (g_2^n)_x\right)_x-\rho_2\,\lambda^2_n y^n=\rho_2\,\lambda^{1-\ell}_n\left(\lambda_n^{-1} h_2^n+ig_2^n\right),\label{E:=(4.8)}\nline
-\kappa \theta_{xx}^n+i\rho_0\,\lambda_n\theta^n-i\gamma \lambda_n  u_{xx}^n=\lambda^{-\ell}_n\left(\rho_0\, h_3^n-\gamma\, (g_1^n)_{xx}\right).\label{E:=(4.9)}
\end{eqnarray}
In addition, we have
\begin{equation}\label{E:=(4.10)}
\beta_1\, u_{xxx}^n+\alpha_1\,\lambda^{-\ell}_n   (h_1^n)_x\in H^1\left(0,L_0\right),\quad
\beta_2\, y_{xxx}^n+\alpha_1\,\lambda^{-\ell}_n   (h_2^n)_x\in H^1\left(L_0,L\right),
\end{equation}
and
\begin{equation}\label{E:=(4.11)}
\begin{array}{lll}

\beta_1\, u_{xxx}^n(L_0)-\beta_2\, y_{xxx}^n(L_0)+ \lambda^2_n\left(\alpha_1  u_x^n(L_0)-\alpha_2 y_x^n(L_0)\right) + \gamma \theta_{x}^n(L_0)

\nline

+\lambda^{-\ell}_n\left(\alpha_1\,   (h_1^n)_x(L_0)-\alpha_2\,  (h_2^n)_x(L_0)\right)
+i\lambda^{1-\ell}_n\left(\alpha_1  (g_1^n)_x(L_0)- \alpha_2 \left(g_2^n\right)_x(L_0) \right)=0.

\end{array}
\end{equation}
\end{Lemma}
\begin{proof} For all $Z=\left(\phi,\varphi,\psi\right)\in \mathbb{W}_2$, using the second equation of \eqref{E:=(4.6)} and equations \eqref{E:=(2.5)} and \eqref{E:=(2.9)},  we get
\begin{equation*}
\begin{array}{lll}

0&=&\left<i\lambda_n\, \Phi_2^n+C^{-1}\left(B\Phi_2^n+A_1\Phi_1^n\right)-\lambda^{-\ell}_n\, h^n,Z\right>_{{\mathbb{V}_2}}

\nline

&=&\left<i\lambda_n\, C\Phi_2^n+B\Phi_2^n+A_1\Phi_1^n-\lambda^{-\ell}_n\, Ch^n,Z\right>_{\mathbb{V}_2'\times {\mathbb{V}_2}}

\nline

&=&i\lambda_n\,  c\left(\Phi_2^n,Z\right)+\left<B\Phi_2^n+A_1\Phi_1^n,Z\right>_{\mathbb{V}_2'\times {\mathbb{V}_2}}-\lambda^{-\ell}_n\, c\left(h^n,Z\right).
\end{array}
\end{equation*}
Since $\Phi^n\in D(\mathcal{A})$; i.e., $\Phi_2^n\in  {\mathbb{W}_2}$ and $\Phi_1^n\in  \mathbb{W}_1$, then using \eqref{E:=(2.5)}, the last equation of \eqref{E:=(2.8)}, and second-third equations of \eqref{E:=(2.8)}  in the above equation, we obtain
\begin{equation*}
\begin{array}{lll}

0&=&i\lambda_n\,  c\left(\Phi_2^n,Z\right)+\left<B\Phi_2^n+A_1\Phi_1^n,Z\right>_{\mathbb{V}_2'\times {\mathbb{V}_2}}-\lambda^{-\ell}_n\, c\left(h^n,Z\right)

\nline

&=&i\lambda_n\,  c\left(\Phi_2^n,Z\right)+\left<B\Phi_2^n+A_1\Phi_1^n,Z\right>_{\mathbb{W}_2'\times {\mathbb{W}_2}}-\lambda^{-\ell}_n\, c\left(h^n,Z\right)

\nline

&=&a\left(\left(u^n,y^n\right),\left(\phi,\varphi\right)\right)+b\left(\left(v^n,z^n,\theta^n\right),\left(\phi,\varphi,\psi\right)\right)+i\lambda_n\,  c\left(\left(v^n,z^n,\theta^n\right),\left(\phi,\varphi,\psi\right)\right)

\nline

&&\ \ -\lambda^{-\ell}_n\, c\left(\left(h_1^n,h_2^n,h_3^n\right),\left(\phi,\varphi,\psi\right)\right).
\end{array}
\end{equation*}
Consequently, from the above equation and \eqref{E:=(2.6)}, we find
\begin{equation}\label{E:=(4.12)}
\begin{array}{lll}

\beta_1\int_0^{L_0}u_{xx}^n\overline{\phi_{xx}}\, dx+\beta_2\int_{L_0}^Ly_{xx}^n\overline{\varphi_{xx}}\, dx+\gamma\int_0^{L_0}\left(v_{x}^n\overline{\psi_x}-\theta_{x}^n\overline{\phi_x}\right) dx+\kappa \int_0^{L_0}\theta_{x}^n\overline{\psi_x}\, dx

 \nline

i\lambda_n \int_0^{L_0}\left(\rho_1\,v^n\overline{\phi}+\alpha_1\, v_x^n\overline{\phi_x}\right)dx+i\lambda_n\int_{L_0}^L\left(\rho_2\,z^n\overline{\varphi}+\alpha_2\, z_x^n\overline{\varphi_x}\right)dx+i\rho_0\,\lambda_n\int_0^{L_0}\theta^n\overline{\psi}\, dx

\nline

=\rho_0\,\lambda^{-\ell}_n \int_0^{L_0}h_3^n\overline{\psi}\, dx+\lambda^{-\ell}_n  \int_0^{L_0}\left(\rho_1\,h_1^n\overline{\phi}+\alpha_1\, (h_1^n)_x\overline{\phi_x}\right)dx

\nline

+\lambda^{-\ell}_n \int_{L_0}^L\left(\rho_2\,h_2^n\overline{\varphi}+\alpha_2\, (h_2^n)_x\overline{\varphi_x}\right)dx,\quad \forall\left(\phi,\varphi,\psi\right)\in \mathbb{W}_2.

\end{array}
\end{equation}
On the other hand, from first equation of \eqref{E:=(4.6)}, we have
\begin{equation}\label{E:=(4.13)}
\left\{
\begin{array}{lll}

v^n=i\lambda_n u^n-\lambda^{-\ell}_n g_1^n\quad\text{in } H^2(0,L_0),

\nline

z^n=i\lambda_n y^n-\lambda^{-\ell}_n g_2^n\quad\text{in } H^2(L_0,L).

\end{array}
\right.
\end{equation}
Inserting the last equations in \eqref{E:=(4.12)}, we infer that
\begin{equation*}
\begin{array}{lll}

\beta_1\int_0^{L_0}u_{xx}^n\overline{\phi_{xx}}\, dx+\beta_2\int_{L_0}^Ly_{xx}^n\overline{\varphi_{xx}}\, dx+\kappa \int_0^{L_0}\theta_{x}^n\overline{\psi_x}\, dx +i\gamma \lambda_n \int_0^{L_0} u_x^n\overline{\psi_x}\,  dx-\gamma \int_0^{L_0}\theta_{x}^n\overline{\phi_x}\, dx

\nline

-\lambda^2_n \int_0^{L_0}\left(\rho_1\, u^n\overline{\phi}+\alpha_1\, u_x^n\overline{\phi_x}\right)dx
-\lambda^2_n\int_{L_0}^L\left(\rho_2\, y^n\overline{\varphi}+\alpha_2\, y_x^n\overline{\varphi_x}\right)dx
+i\rho_0\,\lambda_n\int_0^{L_0}\theta^n\overline{\psi}\, dx =\rho_0\,\lambda^{-\ell}_n \int_0^{L_0}h_3^n\overline{\psi}\, dx

\nline

 +\gamma\, \lambda^{-\ell}_n\int_0^{L_0} (g_1^n)_x\, \overline{\psi_x}\,  dx+i\lambda^{1-\ell}_n\int_0^{L_0}\left(\rho_1\, g_1^n\overline{\phi}+\alpha_1\,  (g_1^n)_x\overline{\phi_x}\right)dx
+i\lambda^{1-\ell}_n\int_{L_0}^L\left(\rho_2\, g_2^n\overline{\varphi}+\alpha_2\, \left(g_2^n\right)_x\overline{\varphi_x}\right)dx

\nline

+\lambda^{-\ell}_n  \int_0^{L_0}\left(\rho_1\, h_1^n\overline{\phi}+\alpha_1\, (h_1^n)_x\overline{\phi_x}\right)dx+\lambda^{-\ell}_n \int_{L_0}^L\left(\rho_2\, h_2^n\overline{\varphi}+\alpha_2\, (h_2^n)_x\overline{\varphi_x}\right)dx,\quad \forall\left(\phi,\varphi,\psi\right)\in \mathbb{W}_2.
\end{array}
\end{equation*}
In the above equation, using integration by parts and equation \eqref{E:=(3.8)}, we obtain
\begin{equation}\label{E:=(4.14)}
\begin{array}{lll}

&-&\int_0^{L_0}\left[\beta_1\, u_{xxx}^n+\alpha_1\,\lambda^{-\ell}_n   (h_1^n)_x\right]\overline{\phi_{x}}\, dx

\nline

&+&\int_0^{L_0}\left[ \lambda^2_n \left(\alpha_1\, u_{xx}^n-\rho_1\, u^n\right)+\gamma  \theta_{xx}^n-\rho_1\,\lambda^{-\ell}_n h_1^n+i\lambda^{1-\ell}_n\left(\alpha_1\,  (g_1^n)_{xx}-\rho_1\, g_1^n\right)\right] \overline{\phi}\,  dx

\nline

&-&\int_{L_0}^L\left[\beta_2\, y_{xxx}^n+\alpha_2\, \lambda^{-\ell}_n (h_2^n)_x\right]\overline{\varphi_{x}}\, dx

\nline

&+&\int_{L_0}^L\left[\lambda^2_n\left( \alpha_2\, y_{xx}^n-\rho_2\, y^n\right)-\rho_2\,\lambda^{-\ell}_n  h_2^n +i\lambda^{1-\ell}_n\left(\alpha_2\, \left(g_2^n\right)_{xx}-\rho_2\, g_2^n\right)\right]\overline{\varphi}\, dx

\nline

&+&\int_0^{L_0}\left[-\kappa \theta_{xx}^n+i\rho_0\, \lambda_n\theta^n-i\gamma \lambda_n  u_{xx}^n-\rho_0\,\lambda^{-\ell}_n h_3^n+\gamma\, \lambda^{-\ell}_n (g_1^n)_{xx}\right]\overline{\psi}\, dx

 \nline

&+&\beta_1\left[u_{xx}^n\overline{\phi_{x}}\right]_0^{L_0}+\beta_2\left[ y_{xx}^n\overline{\varphi_{x}}\right]_{L_0}^L

 \nline

 &-&\left[\left(\alpha_1\lambda^2_n   u_x^n+ \gamma\theta_{x}^n+i\alpha_1\lambda^{1-\ell}_n  (g_1^n)_x \right)\overline{\phi}\right]_0^{L_0}-\left[\left(\alpha_2 \lambda^2_n y_x^n+i\alpha_2\lambda^{1-\ell} \left(g_2^n\right)_x\right)\overline{\varphi}\right]_{L_0}^L

\nline

&+&\left[\left(i\gamma \lambda_n u_x^n+\kappa \theta_{x}^n-\gamma\,  (g_1^n)_{x}\right)\overline{\psi}\right]_0^{L_0}=0,\quad \forall\;\left(\phi,\varphi,\psi\right)\in \mathbb{W}_2.
\end{array}
\end{equation}
Since $\Phi^n\in D\left(\mathcal{A}\right)$ and $\left(\phi,\varphi,\psi\right)\in \mathbb{W}_2$, then from \eqref{E:=(2.12)} and \eqref{E:=(3.9)}, we have the following boundary conditions
\begin{equation*}
\left\{
\begin{array}{lll}
\phi(0)=\phi_x(0)=\varphi(L)=\varphi_x(L)=\psi(0)=\psi(L_0)=0

 \nline

\phi(L_0)=\varphi(L_0),\quad \phi_x(L_0)=\varphi(L_0),\quad \beta_1 u_{xx}^n(L_0)=y_{xx}^n(L_0).

\end{array}
\right.
\end{equation*}
Substituting the above boundary conditions in \eqref{E:=(4.14)}, we derive that
\begin{equation}\label{E:=(4.15)}
\begin{array}{lll}

&-&\int_0^{L_0}\left[\beta_1\, u_{xxx}^n+\alpha_1\,\lambda^{-\ell}_n   (h_1^n)_x\right]\overline{\phi_{x}}\, dx

\nline

&+&\int_0^{L_0}\left[ \lambda^2_n \left(\alpha_1\, u_{xx}^n-\rho_1\, u^n\right)+\gamma \theta_{xx}^n-\lambda^{-\ell}_n \rho_1\, h_1^n+i\lambda^{1-\ell}_n\left(\alpha_1\,  (g_1^n)_{xx}-\rho_1\, g_1^n\right)\right] \overline{\phi}\,  dx

\nline

&-&\int_{L_0}^L\left[\beta_2\, y_{xxx}^n+\alpha_2\, \lambda^{-\ell}_n (h_2^n)_x\right]\overline{\varphi_{x}}\, dx

\nline

&+&\int_{L_0}^L\left[\lambda^2_n\left( \alpha_2\, y_{xx}^n-\rho_2\, y^n\right)-\rho_2\,\lambda^{-\ell}_n  h_2^n +i\lambda^{1-\ell}_n\left(\alpha_2\, \left(g_2^n\right)_{xx}-\rho_2\, g_2^n\right)\right]\overline{\varphi}\, dx

\nline

 &+&\left(\alpha_2 \lambda^2_n y_x^n(L_0)+i\alpha_2\lambda^{1-\ell}_n \left(g_2^n\right)_x(L_0)-\alpha_1\lambda^2   _nu_x(L_0)- \gamma \theta_{x}^n(L_0)-i\alpha_1\lambda^{1-\ell}_n  (g_1^n)_x(L_0)\right)\overline{\phi}(L_0)

\nline

&+&\int_0^{L_0}\left[-\kappa \theta_{xx}^n+i\rho_0\,\lambda_n\theta^n-i\gamma \lambda_n  u_{xx}^n-\rho_0\,\lambda^{-\ell}_n h_3^n+\gamma\, \lambda^{-\ell}_n (g_1^n)_{xx}\right]\overline{\psi}\, dx=0,\quad \forall\;\left(\phi,\varphi,\psi\right)\in \mathbb{W}_2.
\end{array}
\end{equation}
$\bullet$ Taking $\left(\phi,\varphi,\psi\right)=\left(0,0,\psi\right)\in \left\{0\right\}\times \left\{0\right\}\times  C^{1}_c\left(0,L_0\right)\subset \mathbb{W}_2$ in equation \eqref{E:=(4.15)}, one finds
\begin{equation*}
\int_0^{L_0}\left[-\kappa \theta_{xx}^n+i\rho_0\,\lambda_n\theta^n-i\gamma \lambda_n  u_{xx}^n-\rho_0\,\lambda^{-\ell}_n h_3^n+\gamma\, \lambda^{-\ell}_n (g_1)_{xx}^n\right]\overline{\psi}\, dx=0,\quad \forall\;\psi\in  C^{1}_c\left(0,L_0\right),
\end{equation*}
and consequently, we get
\begin{equation*}
-\kappa \theta_{xx}^n+i\rho_0\,\lambda_n\theta^n-i\gamma \lambda_n  u_{xx}^n-\rho_0\,\lambda^{-\ell}_n h_3^n+\gamma\, \lambda^{-\ell}_n (g_1)_{xx}^n=0,\quad\text{a.e. } x\in (0,L_0).
\end{equation*}
Hence, \eqref{E:=(4.9)} holds true.\\[0.1in]
$\bullet$ Taking $\left(\phi,\varphi,\psi\right)=\left(\phi,0,0\right)\in C^{1}_c\left(0,L_0\right)\times \left\{0\right\}\times \left\{0\right\} \subset \mathbb{W}_2$ in equation \eqref{E:=(4.15)}, we get
\begin{equation}\label{E:=(4.16)}
\begin{array}{lll}

\int_0^{L_0}\left[\beta_1\, u_{xxx}^n+\alpha_1\,\lambda^{-\ell}_n   (h_1^n)_x\right]\overline{\phi_{x}}\, dx=

\nline

\int_0^{L_0}\left[ \lambda^2_n \left(\alpha_1\, u_{xx}^n-\rho_1\, u^n\right)+\gamma \theta_{xx}^n-\rho_1\,\lambda^{-\ell}_n h_1^n+i\lambda^{1-\ell}_n\left(\alpha_1\,  (g_1^n)_{xx}-\rho_1\, g_1^n\right)\right] \overline{\phi}\,  dx ,\quad \forall\;\phi\in  C^{1}_c\left(0,L_0\right).

\end{array}
\end{equation}
Since $\lambda^2_n \left(\alpha_1\, u_{xx}^n-\rho_1\, u^n\right)+\gamma \theta_{xx}^n-\rho_1\,\lambda^{-\ell}_n h_1^n+i\lambda^{1-\ell}_n\left(\alpha_1\,  (g_1^n)_{xx}- \rho_1\, g_1^n\right)\in L^2(0,L_0)$,  we get
\begin{equation*}
\beta_1\, u_{xxx}^n+\alpha_1\,\lambda^{-\ell}_n   (h_1^n)_x\in H^1\left(0,L_0\right).
\end{equation*}
Thus, we get the first estimation of \eqref{E:=(4.10)}. Consequently, integrating by parts \eqref{E:=(4.16)}, we obtain
\begin{equation*}
\begin{array}{lll}
\forall\; \phi\in  C^{1}_c\left(0,L_0\right):

\nline

\int_0^{L_0}\left[\left(\beta_1\, u_{xxx}^n+\alpha_1\,\lambda^{-\ell}_n   (h_1^n)_x\right)_x+ \lambda^2_n \left(\alpha_1\, u_{xx}^n-\rho_1\, u^n\right)+\gamma  \theta_{xx}^n-\rho_1\,\lambda^{-\ell}_n h_1^n+i\lambda^{1-\ell}_n\left(\alpha_1\,  (g_1^n)_{xx}-\rho_1\, g_1^n\right)\right] \overline{\phi}\,  dx=0.

\end{array}
\end{equation*}
Thus, one has
\begin{equation*}
\left(\beta_1\, u_{xxx}^n+\alpha_1\,\lambda^{-\ell}_n   (h_1^n)_x\right)_x+ \lambda^2_n \left(\alpha_1\, u_{xx}^n- \rho_1\,u^n\right)+\gamma \theta_{xx}^n-\rho_1\,\lambda^{-\ell}_n h_1^n+i\lambda^{1-\ell}_n\left(\alpha_1\,  (g_1^n)_{xx}-\rho_1\, g_1^n\right)=0, \ \text{a.e. } x\in (0,L_0).
\end{equation*}
Hence,  we derive \eqref{E:=(4.7)}.\\[0.1in]
$\bullet$ By the same way, taking $\left(\phi,\varphi,\psi\right)=\left(0,\varphi,0\right)\in \left\{0\right\}\times C^{1}_c\left(0,L_0\right)\times  \left\{0\right\} \subset \mathbb{W}_2$ in equation \eqref{E:=(4.15)}, one gets
\begin{equation*}
\beta_2\, y_{xxx}^n+\alpha_1\,\lambda^{-\ell}_n   (h_2^n)_x\in H^1\left(L_0,L\right)
\end{equation*}
and
\begin{equation*}
\left(\beta_2\, y_{xxx}^n+\alpha_2\,\lambda^{-\ell}_n   (h_2^n)_x\right)_x+ \lambda^2_n \left(\alpha_2\, y_{xx}^n- \rho_2\, y^n\right)-\rho_2\,\lambda^{-\ell}_n h_2^n+i\lambda^{1-\ell}_n\left(\alpha_2\,  (g_2^n)_{xx}-\rho_2\, g_2^n\right)=0,\quad \text{a.e. } x\in (L_0,L).
\end{equation*}
Thus, we find \eqref{E:=(4.8)} and the second estimation of \eqref{E:=(4.10)}.\\[0.1in]
$\bullet$ Taking $\phi(L_0)=\varphi(L_0)=1$ and $\psi=0$ in \eqref{E:=(4.15)}, then using integration by parts for the first and third terms, we obtain
\begin{equation*}
\begin{array}{lll}

\int_0^{L_0}\left[\left(\beta_1\, u_{xxx}^n+\alpha_1\,\lambda^{-\ell}_n   (h_1^n)_x\right)_x+ \lambda^2_n \left(\alpha_1\, u_{xx}^n-\rho_1\, u^n\right)+\gamma \theta_{xx}^n-\rho_1\,\lambda^{-\ell}_n h_1^n+i\lambda^{1-\ell}_n\left(\alpha_1\,  (g_1^n)_{xx}-\rho_1\, g_1^n\right)\right]\overline{\phi}\, dx

\nline

\int_{L_0}^L\left[\left(\beta_2\, y_{xxx}^n+\alpha_2\, \lambda^{-\ell}_n (h_2^n)_x\right)_x+\lambda^2_n\left( \alpha_2\, y_{xx}^n-\rho_2\,y^n\right)-\rho_2\,\lambda^{-\ell}_n  h_2^n +i\lambda^{1-\ell}_n\left(\alpha_2\, \left(g_2^n\right)_{xx}-\rho_2\, g_2^n\right)\right]\overline{\varphi}\, dx

\nline

-\beta_1\, u_{xxx}^n(L_0)-\alpha_1\,\lambda^{-\ell}_n   (h_1^n)_x(L_0)
+\beta_2\, y_{xxx}^n(L_0)+\alpha_2\, \lambda^{-\ell}_n (h_2^n)_x(L_0)

\nline

+\alpha_2 \lambda^2_n y_x^n(L_0)+i\alpha_2\lambda^{1-\ell}_n \left(g_2^n\right)_x(L_0)-\alpha_1\lambda^2_n  u_x^n(L_0)- \gamma \theta^n_{x}(L_0)-i\alpha_1\lambda^{1-\ell}_n  (g_1^n)_x(L_0)
=0.
\end{array}
\end{equation*}
Substituting \eqref{E:=(4.7)} and \eqref{E:=(4.8)} in the above equation, we obtain \eqref{E:=(4.11)}. The proof is thus complete.
\end{proof}
\noindent From \eqref{E:=(4.4)} and \eqref{E:=(4.13)}, we remark that
\begin{equation}\label{E:=(4.17)}
\left\{
\begin{array}{llll}

\left\|u_{xx}^n\right\|_{L^2(0,L_0)}=O\left(1\right),\quad \left\|\lambda_n\, u_{x}^n\right\|_{L^2(0,L_0)}=O\left(1\right),\quad \left\|\lambda_n\, u_{x}^n\right\|_{L^2(0,L_0)}=O\left(1\right),

\nline

\left\|y_{xx}^n\right\|_{L^2(L_0,L)}=O\left(1\right),\quad \left\|\lambda_n\, y_{x}^n\right\|_{L^2(L_0,L)}=O\left(1\right),\quad \left\|\lambda_n\, y_{x}^n\right\|_{L^2(L_0,L)}=O\left(1\right).

\end{array}
\right.
\end{equation}
From now, we denote by $m_1$ a positive constant number, such that $m_1$ independent of $n$ and $0<m_1<1$, also we denote by $K_j$ a positive constant number independent of  $n$ and $\epsilon_{j,n}$ is a positive number such that $\lim_{\lambda_n\to \infty} \epsilon_{j,n}=0.$
\begin{Lemma}\label{L:=4.4}
For all $\ell\geq0$, the solution $\left(\Phi_1^n,\Phi_2^n\right)\in D(\mathcal{A})$ of System \eqref{E:=(4.6)} satisfies the following asymptotic behavior  estimation
\begin{equation}\label{E:=(4.18)}
\int_0^{L_0}\left|\theta_x^n\right|^2 dx=o\left(\lambda_n^{-\ell}\right)\quad \text{and}\quad \int_0^{L_0}\left|\theta^n\right|^2 dx=o\left(\lambda_n^{-\ell}\right).
\end{equation}
\end{Lemma}
\begin{proof}
Taking the inner product of \eqref{E:=(4.5)}  with $\Phi^n$ in $\mathcal{H}$, then using Cauchy Schwarz inequality, we get
\begin{equation*}
-\Re\left<\mathcal{A}\Phi^n , \Phi^n\right>_{{\mathcal{H}}}=\Re\left<(i \lambda_n I - \mathcal{A})\Phi^n , \Phi^n\right>_{{\mathcal{H}}}  \leq \lambda_n ^{-\ell}\left\|F^n\right\|_{\mathcal{H}} \left\|\Phi^n\right\|_{\mathcal{H}}.
\end{equation*}
Now, similar to Equation \eqref{E:=(3.1)}, we have
\begin{equation}\label{E:=(4.19)}
0\leq\int_0^{L_0}\left|\theta_x^n\right|^2 dx \leq-\frac{1}{\kappa}\Re\left<\mathcal{A} \Phi^n,\Phi^n\right>_{\mathcal{H}}   \leq \epsilon_{1,n}\, \lambda_n ^{-\ell},
\end{equation}
where $\epsilon_{1,n}=\kappa^{-1}\, \left\|F^n\right\|_{\mathcal{H}} \left\|\Phi^n\right\|_{\mathcal{H}}.$  Using \eqref{E:=(4.4)} and \eqref{E:=(4.5)}, we get $\epsilon_{1,n}\to 0$. Hence, from \eqref{E:=(4.19)}, we obtain the first asymptotic estimate of \eqref{E:=(4.18)}. Since $\theta^n\in H^1_0(0,L_0),$ using \eqref{E:=(4.19)} and Poincaré's inequality, we infer that
\begin{equation}\label{E:=(4.20)}
0\leq\int_0^{L_0}\left|\theta^n\right|^2 dx\leq K_p\int_0^{L_0}\left|\theta_x^n\right|^2 dx \leq-\frac{1}{\kappa}\Re\left<\mathcal{A} \Phi^n,\Phi^n\right>_{\mathcal{H}}  \leq \epsilon_{2,n}\, \lambda_n ^{-\ell},
\end{equation}
where $K_p$ is the Poincaré constant and $\epsilon_{2,n}=K_p\epsilon_{1,n}\to 0.$ Hence, from \eqref{E:=(4.20)}, we get the second asymptotic estimate of  \eqref{E:=(4.18)}. The proof is thus complete.
\end{proof}
\begin{Lemma}\label{L:=4.5}
For all $\ell\geq0$, the solution $\left(\Phi_1^n,\Phi_2^n\right)\in D(\mathcal{A})$ of System \eqref{E:=(4.6)} satisfies the following asymptotic behavior  estimations
\begin{eqnarray}
\int_0^{L_0}\left|u_{xx}^n\right|^2 dx=o\left(\lambda^{-\ell}_n\right),\label{E:=(4.21)}\nline
\left|u_{xx}^n\right|_{\infty}=o\left(\lambda^{\frac{1-\ell}{2}}_n\right).\label{E:=(4.22)}
\end{eqnarray}
\end{Lemma}
 \noindent For the proof of Lemma \ref{L:=4.5}, we need the following lemmas.
\begin{Lemma}\label{L:=4.6}
For all $\ell\geq0$, the solution $\left(\Phi_1^n,\Phi_2^n\right)\in D(\mathcal{A})$ of System \eqref{E:=(4.6)} satisfies the following asymptotic behavior  estimation
\begin{equation}\label{E:=(4.23)}
\lambda^{-2}_n\left\|u_{xxx}^n\right\|_{L^2(0,L_0)}^2\leq K_1\left(1+ \lambda^{-1}_n\right)^2 \left\| u_{xx}^n\right\|_{L^2(0,L_0)}^2+\epsilon_{3,n}\, \lambda^{-\ell}_n.
\end{equation}
\end{Lemma}
\begin{proof} First, since $u_{xx}^n+\frac{\alpha_1}{\beta_1}\lambda^{-\ell}_n h_1^n\in H^2(0,L_0)$, then applying \eqref{E:=(A.1)}, we obtain
\begin{equation*}
\begin{array}{lll}
\left\|u_{xxx}^n+\frac{\alpha_1}{\beta_1}\lambda^{-\ell}_n (h_1^n)_x\right\|_{L^2(0,L_0)}\leq K_6 \left\|u_{xxxx}^n+\frac{\alpha_1}{\beta_1}\lambda^{-\ell}_n (h_1^n)_{xx}\right\|^{\frac{1}{2}}_{L^2(0,L_0)} \left\|u_{xx}^n+\frac{\alpha_1}{\beta_1}\lambda^{-\ell}_n h_1^n \right\|^{\frac{1}{2}}_{L^2(0,L_0)}

\nline

\hspace{5cm} +K_7 \left\|u_{xx}^n+\frac{\alpha_1}{\beta_1}\lambda^{-\ell}_n h_1^n\right\|_{L^2(0,L_0)}.

 \end{array}
\end{equation*}
Thus
\begin{equation*}
\begin{array}{lll}
\left\|u_{xxx}^n\right\|_{L^2(0,L_0)}\leq K_6 \left\|u_{xxxx}^n+\frac{\alpha_1}{\beta_1}\lambda^{-\ell}_n (h_1^n)_{xx}\right\|^{\frac{1}{2}}_{L^2(0,L_0)} \left\|u_{xx}^n+\frac{\alpha_1}{\beta_1}\lambda^{-\ell}_n h_1^n \right\|^{\frac{1}{2}}_{L^2(0,L_0)}

\nline

\hspace{5cm} +K_7 \left\|u_{xx}^n\right\|_{L^2(0,L_0)}+\frac{\alpha_1}{\beta_1}\left(K_7+1\right)\lambda^{-\ell}_n\left\| (h_1^n)_x\right\|_{L^2(0,L_0)},

 \end{array}
\end{equation*}
In the above inequality, using the first estimation of \eqref{E:=(A.6)}, we get
\begin{equation*}
\begin{array}{lll}
\left\|u_{xxx}^n\right\|_{L^2(0,L_0)}^2\leq 3K_6^2 \left\|u_{xxxx}^n+\frac{\alpha_1}{\beta_1}\lambda^{-\ell}_n (h_1^n)_{xx}^n\right\|_{L^2(0,L_0)} \left\|u_{xx}^n+\frac{\alpha_1}{\beta_1}\lambda^{-\ell}_n h_1^n \right\|_{L^2(0,L_0)}

\nline

\hspace{5cm} +3K_7^2 \left\|u_{xx}^n\right\|_{L^2(0,L_0)}^2+3\alpha_1^2 \beta_2^{-2}\left(K_7+1\right)^2\lambda^{-2\ell}_n\left\| (h_1^n)_x\right\|_{L^2(0,L_0)}^2.

 \end{array}
\end{equation*}
Consequently, one derives
\begin{equation}\label{E:=(4.24)}
\begin{array}{lll}
\lambda^{-2}_n\left\|u_{xxx}^n\right\|_{L^2(0,L_0)}^2\leq 3K_6^2 \lambda^{-2}_n\left\|u_{xxxx}^n+\frac{\alpha_1}{\beta_1}\lambda^{-\ell}_n (h_1^n)_{xx}^n\right\|_{L^2(0,L_0)} \left\|u_{xx}^n\right\|_{L^2(0,L_0)}+3K_7^2 \lambda^{-2}\left\|u_{xx}^n\right\|_{L^2(0,L_0)}^2

\nline

\hspace{2.8cm}+3K_6^2 \alpha_1\beta_1^{-1}\lambda^{-\ell-2}_n \left\|u_{xxxx}^n+\frac{\alpha_1}{\beta_1}\lambda^{-\ell}_n (h_1^n)_{xx}^n\right\|_{L^2(0,L_0)} \left\| h_1^n \right\|_{L^2(0,L_0)}

\nline

\hspace{2.8cm}+3\alpha_1^2 \beta_2^{-2}\left(K_7+1\right)^2\lambda^{-2\ell-2}_n\left\| (h_1^n)_x\right\|_{L^2(0,L_0)}^2.

 \end{array}
\end{equation}
Next,  from \eqref{E:=(4.7)} and \eqref{E:=(4.9)}, we obtain
\begin{equation}\label{E:=(4.25)}
\begin{array}{llll}

\lambda^{-2}_n\left\|\left( u_{xxx}^n+\frac{\alpha_1}{\beta_1}\,\lambda^{-\ell}_n   (h_1^n)_x\right)_x\right\|_{L^2(0,L_0)}\leq  \alpha_1 \beta_1^{-1} \left\| u_{xx}^n\right\|_{L^2(0,L_0)}+\beta_1^{-1}\rho_1\left\| u^n\right\|_{L^2(0,L_0)}

\nline

+\gamma \beta_1^{-1}\lambda^{-2}_n \left\|\theta_{xx}^n\right\|_{L^2(0,L_0)}+\beta_1^{-1}\lambda^{-1-\ell}_n\left\|\lambda^{-1}_n\rho_1 h_1^n+i\left(\rho_1g_1^n-\alpha_1\,  (g_1^n)_{xx}\right)\right\|_{L^2(0,L_0)},
\end{array}
\end{equation}
and
\begin{equation}\label{E:=(4.26)}
\begin{array}{llll}
 \gamma \beta_1^{-1}\lambda^{-2}_n\left\|\theta_{xx}^n\right\|_{L^2(0,L_0)}\leq \gamma^2 \beta_1^{-1}\kappa^{-1}  \lambda^{-1}_n \left\|u_{xx}^n\right\|_{L^2(0,L_0)}+
 \gamma \beta_1^{-1}\kappa^{-1}\rho_0\lambda^{-1}_n\left\|\theta^n\right\|_{L^2(0,L_0)}

 \nline

 \hspace{3cm}
 +\gamma \beta_1^{-1}\kappa^{-1}\lambda^{-2-\ell}_n\left\|\rho_0 h_3^n-\gamma\, (g_1^n)_{xx}\right\|_{L^2(0,L_0)}.
 \end{array}
\end{equation}
Now, substituting \eqref{E:=(4.26)} in \eqref{E:=(4.25)}, we find
\begin{equation}\label{E:=(4.27)}
\begin{array}{llll}

\lambda^{-2}_n\left\|\left( u_{xxx}^n+\frac{\alpha_1}{\beta_1}\,\lambda^{-\ell}_n   (h_1^n)_x\right)_x\right\|_{L^2(0,L_0)}\leq  \left(\alpha_1 +\gamma^2 \kappa^{-1}\lambda^{-1}_n\right)\beta_1^{-1}\left\| u_{xx}^n\right\|_{L^2(0,L_0)}+\beta_1^{-1}\rho_1\left\| u^n\right\|_{L^2(0,L_0)}

\nline

+ \gamma \beta_1^{-1}\kappa^{-1}\rho_0\lambda^{-1}_n\left\|\theta^n\right\|_{L^2(0,L_0)}+\beta_1^{-1}\lambda^{-1-\ell}_n\left\|\lambda^{-1}_n\rho_1 h_1^n+i\left(\rho_1g_1^n-\alpha_1\,  (g_1^n)_{xx}\right)\right\|_{L^2(0,L_0)}

\nline

+\gamma \beta_1^{-1}\kappa^{-1}\lambda^{-2-\ell}_n\left\|\rho_0 h_3^n-\gamma\, (g_1^n)_{xx}\right\|_{L^2(0,L_0)}.
\end{array}
\end{equation}
Since $u\in H^3(0,L_0)$ with $u(0)=u_x(0)=0$, using Poincar\'{e}'s inequality, we obtain that there exists $\tilde{K}_p>0$ independent of $n$, such that
\begin{equation*}
\left\| u^n\right\|_{L^2(0,L_0)}\leq \tilde{K}_p \left\| u_{xx}^n\right\|_{L^2(0,L_0)}.
\end{equation*}
Inserting the above inequality into \eqref{E:=(4.27)},  we derive
\begin{equation}\label{E:=(4.28)}
\begin{array}{llll}

\lambda^{-2}_n\left\|\left( u_{xxx}^n+\frac{\alpha_1}{\beta_1}\,\lambda^{-\ell}_n   (h_1^n)_x\right)_x\right\|_{L^2(0,L_0)}\leq  \left(\alpha_1 +\rho_1\tilde{K}_p+\gamma^2 \kappa^{-1}\lambda^{-1}_n\right)\beta_1^{-1} \left\| u_{xx}^n\right\|_{L^2(0,L_0)}

\nline

\hspace{6cm}+\gamma \beta_1^{-1}\kappa^{-1}\rho_0\lambda^{-1}_n\left\|\theta^n\right\|_{L^2(0,L_0)}+\epsilon_{4,n}\lambda^{-1-\ell}_n,

\end{array}
\end{equation}
where
\begin{equation*}
\epsilon_{4,n}=\beta_1^{-1}\left\|\rho_1\lambda^{-1}_n h_1^n+i\left(\rho_1g_1^n-\alpha_1\,  (g_1^n)_{xx}\right)\right\|_{L^2(0,L_0)}+\gamma \beta_1^{-1}\kappa^{-1}\lambda^{-1}_n\left\|\rho_0 h_3^n-\gamma\, (g_1^n)_{xx}\right\|_{L^2(0,L_0)}.
\end{equation*}
Since $h_1^n\to 0,\, h_3^n\to 0, \, g_1^n\to0$  and $(g_1^n)_{xx}\to 0$ in $L^2(0,L_0)$, then $\epsilon_{4,n}\to 0$. Substituting \eqref{E:=(4.28)} in \eqref{E:=(4.24)}, we get
\begin{equation}\label{E:=(4.29)}
\begin{array}{lll}
\lambda^{-2}_n\left\|u_{xxx}^n\right\|_{L^2(0,L_0)}^2\leq \left[3K_6^2 \left(\alpha_1 +\rho_1\tilde{K}_p+\gamma^2 \kappa^{-1}\lambda^{-1}_n\right)\beta_1^{-1}+3K_7^2 \lambda^{-2}_n\right] \left\| u_{xx}^n\right\|_{L^2(0,L_0)}^2

\nline

+3K_6^2 \gamma \beta_1^{-1}\kappa^{-1}\rho_0\lambda^{-1}_n\left\|\theta^n\right\|_{L^2(0,L_0)}\left\|u_{xx}^n\right\|_{L^2(0,L_0)}+\epsilon_{5,n}\, \lambda^{-\ell}_n ,
 \end{array}
\end{equation}
where
\begin{equation*}
\begin{array}{lll}

\epsilon_{5,n}=\left[3K_6^2 \epsilon_{4,n}\lambda^{-1}_n+3K_6^2 \alpha_1\beta_1^{-1}\left(\alpha_1 +\rho_1\tilde{K}_p+\gamma^2 \kappa^{-1}\lambda^{-1}_n\right)\beta_1^{-1}  \left\| h_1^n \right\|_{L^2(0,L_0)}\right]\left\|u_{xx}^n\right\|_{L^2(0,L_0)}

\nline

+3K_6^2 \alpha_1\beta_1^{-1} \gamma \beta_1^{-1}\kappa^{-1}\rho_0\lambda^{-\ell-1}_n\left\|\theta^n\right\|_{L^2(0,L_0)}\left\| h_1^n \right\|_{L^2(0,L_0)}

\nline

+3K_6^2 \alpha_1\beta_1^{-1}\epsilon_{4,n}\, \lambda^{-2\ell-1}_n  \left\| h_1^n \right\|_{L^2(0,L_0)}
+3\alpha_1^2 \beta_2^{-2}\left(K_7+1\right)^2\lambda^{-\ell-2}_n\left\| (h_1^n)_x\right\|_{L^2(0,L_0)}^2.

\end{array}
\end{equation*}
From \eqref{E:=(4.17)}, \eqref{E:=(4.20)},  the fact that $h_1^n\to 0$ and  $(h_1^n)_x\to 0$ in $L^2(0,L_0)$, we obtain $\epsilon_{5,n}\to 0.$ Next, taking $p=\lambda^{-1}_n\left\|u_{xx}^n\right\|_{L^2(0,L_0)}$ and $q=3K_6^2 \gamma \beta_1^{-1}\kappa^{-1}\rho_0\left\|\theta^n\right\|_{L^2(0,L_0)}$ in \eqref{E:=(A.3)}, then using \eqref{E:=(4.20)},  one gets
\begin{equation*}
\begin{array}{lll}

3K_6^2 \gamma \beta_1^{-1}\kappa^{-1}\rho_0\lambda^{-1}_n\left\|\theta^n\right\|_{L^2(0,L_0)}\left\|u_{xx}^n\right\|_{L^2(0,L_0)}&\leq  \frac{\lambda^{-2}_n\left\|u_{xx}^n\right\|^2}{2}+\frac{9K_6^4 \gamma^2 \beta_1^{-2}\kappa^{-2}\rho_0^2\left\|\theta^n\right\|_{L^2(0,L_0)}^2}{2}

\nline

&\leq  \frac{\lambda^{-2}_n\left\|u_{xx}^n\right\|^2}{2}+\frac{9K_6^4 \gamma^2 \beta_1^{-2}\kappa^{-2}\rho_0^2\epsilon_{2,n}}{2}\, \lambda_n ^{-\ell}.

\end{array}
\end{equation*}
Inserting the above inequality in \eqref{E:=(4.29)}, we have
\begin{equation*}
\lambda^{-2}_n\left\|u_{xxx}^n\right\|_{L^2(0,L_0)}^2\leq \left[3K_6^2 \left(\alpha_1 +\rho_1\tilde{K}_p+\gamma^2 \kappa^{-1}\lambda^{-1}_n\right)\beta_1^{-1}+\left(3K_7^2+\frac{1}{2}\right) \lambda^{-2}\right] \left\| u_{xx}\right\|_{L^2(0,L_0)}^2+\epsilon_{3,n}\, \lambda^{-\ell}_n,
\end{equation*}
where $\epsilon_{3,n}=\epsilon_{5,n}+\frac{9K_6^4 \gamma^2 \beta_1^{-2}\kappa^{-2}\rho_0\epsilon_{2,n}}{2}\to 0$. In the above inequality, let
\begin{equation*}
K_1= \max\left(3K_6^2\beta_1^{-1}\left(\alpha_1 +\rho_1\tilde{K}_p\right),3K_6^2\beta_1^{-1}\gamma^2 \kappa^{-1},3K_7^2+\frac{1}{2}\right),
\end{equation*}
it holds that
\begin{equation*}
\begin{array}{lll}
\lambda^{-2}_n\left\|u_{xxx}^n\right\|_{L^2(0,L_0)}^2\leq

  K_1\left(1+\lambda^{-1}_n+ \lambda^{-2}_n\right) \left\| u_{xx}^n\right\|_{L^2(0,L_0)}^2+\epsilon_{3,n}\, \lambda^{-\ell}_n

 \leq K_1\left(1+ \lambda^{-1}_n\right)^2 \left\| u_{xx}^n\right\|_{L^2(0,L_0)}^2+\epsilon_{3,n}\, \lambda^{-\ell}_n.

\end{array}
\end{equation*}
Hence, we get \eqref{E:=(4.23)}.
\end{proof}
\begin{Lemma}\label{L:=4.7}
For all $\ell\geq0$, the solution $\left(\Phi_1^n,\Phi_2^n\right)\in D(\mathcal{A})$ of System \eqref{E:=(4.6)} satisfies the following asymptotic behavior  estimation
\begin{equation}\label{E:=(4.30)}
K_2 \lambda^{-4}_n\left\|\theta_{x}^n\right\|_{L^2(0,L_0)}^4\left\|u_{xxx}^n\right\|_{L^2(0,L_0)}^4\leq m_1 \left\|u_{xx}^n\right\|_{L^2(0,L_0)}^8+\epsilon_{6,n}\, \lambda^{-4\ell}_n,
\end{equation}
where $K_2=64\kappa^4\gamma^{-4}$.
\end{Lemma}
\begin{proof} First, using the first estimation of \eqref{E:=(A.6)} for \eqref{E:=(4.23)}, we get
\begin{equation*}
\lambda^{-4}_n\left\|u_{xxx}^n\right\|_{L^2(0,L_0)}^4\leq 2K_1^2\left(1+ \lambda^{-1}_n\right)^4 \left\| u_{xx}^n\right\|_{L^2(0,L_0)}^4+2\epsilon_{3,n}^2\lambda^{-2\ell}_n.
\end{equation*}
Consequently, we  have
\begin{equation}\label{E:=(4.31)}
\begin{array}{lll}

K_2\lambda^{-4}_n\left\|\theta_{x}^n\right\|_{L^2(0,L_0)}^4\left\|u_{xxx}^n\right\|_{L^2(0,L_0)}^4\leq 2K_1^2K_2\left(1+ \lambda^{-1}_n\right)^4 \left\|\theta_{x}^n\right\|_{L^2(0,L_0)}^4\left\| u_{xx}^n\right\|_{L^2(0,L_0)}^4

\nline

\hspace{5.5cm}+2\epsilon_{3,n}^2\lambda^{-2\ell}_n \left\|\theta_{x}^n\right\|_{L^2(0,L_0)}^4.

\end{array}
\end{equation}
From \eqref{E:=(4.19)}, we  obtain
\begin{equation}\label{E:=(4.32)}
2\epsilon_{3,n}^2\lambda^{-2\ell}_n \left\|\theta_{x}^n\right\|_{L^2(0,L_0)}^4\leq 2\epsilon_{1,n}^2 \epsilon_{3,n}^2\,  \lambda^{-4\ell}_n.
\end{equation}
Next, taking $a=\left\| u_{xx}^n\right\|_{L^2(0,L_0)}^4$ and $b=2K_1^2K_2\left(1+ \lambda^{-1}_n\right)^4 \left\|\theta_{x}^n\right\|_{L^2(0,L_0)}^4$ in \eqref{E:=(A.4)}, then using \eqref{E:=(4.19)}, we  derive
\begin{equation}\label{E:=(4.33)}
\begin{array}{lll}

2K_1^2K_2\left(1+ \lambda^{-1}_n\right)^4 \left\|\theta_{x}^n\right\|_{L^2(0,L_0)}^4\left\| u_{xx}^n\right\|_{L^2(0,L_0)}^4&\leq m_1 \left\| u_{xx}^n\right\|_{L^2(0,L_0)}^8+\frac{K_1^4K_2^2\left(1+ \lambda^{-1}_n\right)^8 \left\|\theta_{x}^n\right\|_{L^2(0,L_0)}^8}{m_1}

\nline

&\leq m_1 \left\| u_{xx}^n\right\|_{L^2(0,L_0)}^8+\frac{K_1^4K_2^2\left(1+ \lambda^{-1}_n\right)^8  \epsilon_{1,n}^4}{m_1}\lambda^{-4\ell}_n.

\end{array}
\end{equation}
Substituting \eqref{E:=(4.32)} and \eqref{E:=(4.33)} in \eqref{E:=(4.31)}, we get \eqref{E:=(4.30)}, where
$\epsilon_{6,n}=2\epsilon_{1,n}^2 \epsilon_{3,n}^2+\frac{K_1^4K_2^2\left(1+ \lambda^{-1}_n\right)^8  \epsilon_{1,n}^4}{m_1}\to 0$.
\end{proof}
\begin{Lemma}\label{L:=4.8}
For all $\ell\geq0$, the solution $\left(\Phi_1^n,\Phi_2^n\right)\in D(\mathcal{A})$ of System \eqref{E:=(4.6)} satisfies the following asymptotic behavior  estimation
\begin{equation}\label{E:=(4.34)}
\lambda^{-2}_n\left(\left|\theta_{x}^n(L_0)\right|^2+\left|\theta_{x}^n(L_0)\right|^2 \right)^2\leq K_3 \left\|\theta_{x}^n\right\|_{L^2(0,L_0)}^2\left\|u_{xx}^n\right\|_{L^2(0,L_0)}^2+\epsilon_{7,n}\,\lambda^{-2\ell}_n,
\end{equation}
where $K_3=16\gamma^2\kappa^{-2}.$
\end{Lemma}
\begin{proof} First, let's take
\begin{equation*}
P(x)=\cos\left(\frac{L_0-x}{L_0}\pi\right).
\end{equation*}
Then, we have
\begin{equation}\label{E:=(4.35)}
P(0)=-1,\quad P(L_0)=1,\quad \left|P\right|_{\infty}=1,\quad \left|P'\right|_{\infty}=\pi\, L_0^{-1}.
\end{equation}
Next, from  \eqref{E:=(4.9)}, we have
\begin{equation*}
 \theta_{xx}^n=i\kappa^{-1}\rho_0\lambda_n\theta^n-i\gamma\kappa^{-1} \lambda_n  u_{xx}^n-\kappa^{-1}\lambda^{-\ell}_n\left(\rho_0h_3^n-\gamma\, (g_1^n)_{xx}\right).
\end{equation*}
Multiplying the above equation by $2 P\, \lambda_n^{-1}\, \theta_x^n$ in $L^2(0,L_0)$,  taking the real parts, then using integration by parts and \eqref{E:=(4.35)},  we get
\begin{equation*}
\begin{array}{lll}

\lambda_n^{-1}\left(\left|\theta_{x}^n(L_0)\right|^2+\left|\theta_{x}^n(L_0)\right|^2\right)=\lambda_n^{-1}\int_0^{L_0}P'\left|\theta_{x}^n\right|^2 dx+2\kappa^{-1}\rho_0\Re\left\{i \int_0^{L_0}P\theta^n\, \overline{\theta_x^n}\, dx\right\}

 \nline

 -2\gamma\kappa^{-1} \Re\left\{i\int_0^{L_0}  P  u_{xx}^n\overline{\theta_x^n}\, dx\right\}-2\kappa^{-1}\lambda^{-\ell-1}_n\lambda\Re\left\{ \int_0^{L_0}P\left(\rho_0h_3^n-\gamma\, (g_1^n)_{xx}\right)\overline{\theta_x^n}\, dx\right\}.

  \end{array}
\end{equation*}
Consequently, we  obtain
\begin{equation*}
\begin{array}{lll}

\lambda_n^{-1}\left(\left|\theta_{x}^n(L_0)\right|^2+\left|\theta_{x}^n(L_0)\right|^2\right)\leq \pi\, L_0^{-1}\lambda_n^{-1}\left\|\theta_{x}^n\right\|_{L^2(0,L_0)}^2 +2\kappa^{-1}\rho_0 \left\|\theta_{x}^n\right\|_{L^2(0,L_0)}\left\|\theta^n\right\|_{L^2(0,L_0)}

 \nline

 +2\gamma\kappa^{-1}  \left\|\theta_{x}^n\right\|_{L^2(0,L_0)}\left\|u_{xx}^n\right\|_{L^2(0,L_0)}+2\kappa^{-1}\lambda^{-\ell-1}_n\left\|\rho_0h_3^n-\gamma\, (g_1^n)_{xx}\right\|_{L^2(0,L_0)} \left\|\theta_{x}^n\right\|_{L^2(0,L_0)}.

  \end{array}
\end{equation*}
Thus, using \eqref{E:=(4.35)} and the first estimation of \eqref{E:=(A.6)} in the above  inequality,  one has
\begin{equation*}
\begin{array}{lll}

\lambda_n^{-2}\left(\left|\theta_{x}^n(L_0)\right|^2+\left|\theta_{x}^n(L_0)\right|^2 \right)^2\leq 16\gamma^2\kappa^{-2}  \left\|\theta_{x}^n\right\|_{L^2(0,L_0)}^2\left\|u_{xx}^n\right\|_{L^2(0,L_0)}^2+4\pi^2 L_0^{-2}\lambda_n^{-2}\left\|\theta_{x}^n\right\|_{L^2(0,L_0)}^4\nline

+6\kappa^{-2}\rho_0^2 \left\|\theta_{x}^n\right\|_{L^2(0,L_0)}^2\left\|\theta^n\right\|_{L^2(0,L_0)}^2+16\kappa^{-2}\lambda^{-2\ell-2}_n\left\|\rho_0h_3^n-\gamma\, (g_1^n)_{xx}\right\|_{L^2(0,L_0)}^2 \left\|\theta_{x}^n\right\|_{L^2(0,L_0)}^2.
  \end{array}
\end{equation*}
Consequently, we get
\begin{equation}\label{E:=(4.36)}
\lambda^{-2}_n\left(\left|\theta_{x}^n(L_0)\right|^2+\left|\theta_{x}^n(L_0)\right|^2 \right)^2\leq K_3 \left\|\theta_{x}^n\right\|_{L^2(0,L_0)}^2\left\|u_{xx}^n\right\|_{L^2(0,L_0)}^2+\epsilon_{7,n}\,\lambda^{-2\ell}_n,
\end{equation}
where
\begin{equation*}
\begin{array}{lll}
\epsilon_{7,n}=4\pi^2 L_0^{-2}\lambda_n^{-2+2\ell}\left\|\theta_{x}^n\right\|_{L^2(0,L_0)}^4

+6\kappa^{-2}\rho_0^2 \lambda_n^{2\ell}\left\|\theta_{x}^n\right\|_{L^2(0,L_0)}^2\left\|\theta^n\right\|_{L^2(0,L_0)}^2

\nline

\hspace{1cm}+16\kappa^{-2}\lambda^{-2}_n\left\|\rho_0h_3^n-\gamma\, (g_1^n)_{xx}\right\|_{L^2(0,L_0)}^2 \left\|\theta_{x}^n\right\|_{L^2(0,L_0)}^2.

  \end{array}
\end{equation*}
Using \eqref{E:=(4.19)}-\eqref{E:=(4.20)} and the fact that $h_3^n-\gamma\, (g_1)_{xx}^n\to 0$  in  $L^2(0,L_0)$, we  find
\begin{equation*}
0\leq \epsilon_{7,n}\leq 4\pi^2 L_0^{-2}\lambda_n^{-2}\epsilon_{1,n}^2+6\kappa^{-2} \epsilon_{1,n} \epsilon_{2,n}+16\kappa^{-2}\lambda^{-2-\ell}_n\epsilon_{1,n}\left\|h_3^n-\gamma\, (g_1^n)_{xx}\right\|_{L^2(0,L_0)}^2\to 0.
\end{equation*}
 Hence, from \eqref{E:=(4.36)}, we get \eqref{E:=(4.48)}.
\end{proof}
\begin{Lemma}\label{L:=4.9}
For all $\ell\geq0$, the solution $\left(\Phi_1^n,\Phi_2^n\right)\in D(\mathcal{A})$ of System \eqref{E:=(4.6)} satisfies the following asymptotic behavior  estimation
\begin{equation}\label{E:=(4.37)}
\lambda^{-2}_n\left|u_{xx}^n\right|_{\infty}^4\leq  K_4\left(1+\lambda^{-1}_n\right)^2 \left\| u_{xx}^n\right\|_{L^2(0,L_0)}^4 +\epsilon_{8,n}\,\lambda^{-2\ell}_n,
\end{equation}
\end{Lemma}
\begin{proof} First, since $u_{xx}^n\in  H^1(0,L_0)$, then applying \eqref{E:=(A.2)}, we obtain
\begin{equation*}
\left|u_{xx}^n\right|_{\infty}\leq K_8 \left\|u_{xxx}^n\right\|^{\frac{1}{2}}_{L^2(0,L_0)} \left\|u_{xx}^n\right\|^{\frac{1}{2}}_{L^2(0,L_0)} +K_9 \left\|u_{xx}^n\right\|_{L^2(0,L_0)},
\end{equation*}
consequently, using the second estimation of \eqref{E:=(A.6)}, one has
\begin{equation}\label{E:=(4.38)}
\left|u_{xx}^n\right|_{\infty}^4\leq 8K_8^4 \left\|u_{xxx}^n\right\|^{2}_{L^2(0,L_0)} \left\|u_{xx}^n\right\|^{2}_{L^2(0,L_0)} +8K_9^4 \left\|u_{xx}^n\right\|_{L^2(0,L_0)}^4.
\end{equation}
Substituting \eqref{E:=(4.23)} in \eqref{E:=(4.38)}, we  derive
\begin{equation}\label{E:=(4.39)}
\lambda_n^{-2}\left|u_{xx}^n\right|_{\infty}^4\leq \left[8K_8^4 K_1\left(1+ \lambda^{-1}_n\right)^2+8K_9^4\lambda_n^{-2} \right] \left\| u_{xx}^n\right\|_{L^2(0,L_0)}^4+8K_8^4 \epsilon_{3,n}\, \lambda^{-\ell}_n   \left\| u_{xx}^n\right\|_{L^2(0,L_0)}^2 .
\end{equation}
Taking $p=\left\|u_{xx}^n\right\|_{L^2(0,L_0)}$ and $q=8K_8^4 \epsilon_{3,n}\, \lambda^{-\ell}_n$ in \eqref{E:=(A.3)}, we  obtain
\begin{equation*}
\epsilon_{3,n}\lambda^{-\ell}_n \left\|u_{xx}^n\right\|^{2}_{L^2(0,L_0)}\leq \frac{\left\|u_{xx}^n\right\|_{L^2(0,L_0)}^4}{2}+32K_8^8 \epsilon_{3,n}^2\,\lambda^{-2\ell}_n.
\end{equation*}
Substituting the above equation in \eqref{E:=(4.39)}, we  see that
\begin{equation*}
\lambda_n^{-2}\left|u_{xx}^n\right|_{\infty}^4\leq \left[8K_8^4 K_1\left(1+ \lambda^{-1}_n\right)^2+8K_9^4\lambda_n^{-2}+\frac{1}{2} \right] \left\| u_{xx}^n\right\|_{L^2(0,L_0)}^4+32K_8^8 \epsilon_{3,n}^2\,\lambda^{-2\ell}_n.
\end{equation*}
 Hence  \eqref{E:=(4.37)}  holds true, with $K_4=2\max\left[8K_8^4 K_1,8K_9^4,\frac{1}{2}\right]$ and $\epsilon_{8,n}=32K_8^8 \epsilon_{3,n}^2\to 0.$
\end{proof}
\begin{Lemma}\label{L:=4.10}
For all $\ell\geq0$, the solution $\left(\Phi_1^n,\Phi_2^n\right)\in D(\mathcal{A})$ of System \eqref{E:=(4.6)} satisfies the following asymptotic behavior  estimation
\begin{equation}\label{E:=(4.40)}
K_5 \lambda^{-4}_n\left(\left|\theta_{x}^n\left(L_0\right)\right|^2+\left|\theta_{x}^n\left(0\right)\right|^2\right)^2 \left| {u_{xx}^n}\right|_{\infty}^4\leq 4m_1 \left\| u_{xx}^n\right\|_{L^2(0,L_0)}^8+\epsilon_{9,n}\, \lambda^{-4\ell}_n,
\end{equation}
where $K_5=256\kappa^4\gamma^{-4}.$
\end{Lemma}
\begin{proof} First, from  \eqref{E:=(4.37)} and \eqref{E:=(4.34)}, we get
\begin{equation}\label{E:=(4.41)}
\begin{array}{llll}

K_5 \lambda^{-4}_n\left(\left|\theta_{x}^n\left(L_0\right)\right|^2+\left|\theta_{x}^n\left(0\right)\right|^2\right)^2 \left| {u_{xx}^n}\right|_{\infty}^4\leq K_3 K_5 \left\|\theta_{x}^n\right\|_{L^2(0,L_0)}^2\left\|u_{xx}^n\right\|_{L^2(0,L_0)}^2\epsilon_{8,n}\, \lambda^{-2\ell}_n

\nline

+K_3   K_5 K_4\left(1+\lambda^{-1}_n\right)^2\left\|\theta_{x}^n\right\|_{L^2(0,L_0)}^2\left\|u_{xx}^n\right\|_{L^2(0,L_0)}^2 \left\| u_{xx}^n\right\|_{L^2(0,L_0)}^4

\nline

+K_5K_4\left(1+\lambda^{-1}_n\right)^2\epsilon_{7,n}\lambda^{-2\ell}_n  \left\| u_{xx}^n\right\|_{L^2(0,L_0)}^4+K_5\epsilon_{7,n} \epsilon_{8,n}\lambda^{-4\ell}_n.

\end{array}
\end{equation}
Now, taking $a=\left\| u_{xx}^n\right\|_{L^2(0,L_0)}^4$ and $K_5K_4\left(1+\lambda^{-1}_n\right)^2\epsilon_{7,n}\lambda^{-2\ell}_n $ in  \eqref{E:=(A.3)}, we  obtain
\begin{equation}\label{E:=(4.42)}
K_5K_4\left(1+\lambda^{-1}_n\right)^2\epsilon_{7,n}\, \lambda^{-2\ell}_n  \left\| u_{xx}^n\right\|_{L^2(0,L_0)}^4\leq m_1 \left\| u_{xx}^n\right\|_{L^2(0,L_0)}^8+\frac{K_5^2K_4^2\left(1+\lambda^{-1}_n\right)^4\epsilon_{7,n}^2}{4m_1} \lambda^{-4\ell}_n.
\end{equation}
Next, taking $c=K_3   K_5 K_4\left(1+\lambda^{-1}_n\right)^2\left\|\theta_{x}^n\right\|_{L^2(0,L_0)}^2$, $b=\left\|u_{xx}^n\right\|_{L^2(0,L_0)}^2 $, and $a=\left\| u_{xx}^n\right\|_{L^2(0,L_0)}^4$ in \eqref{E:=(A.5)}, then using  \eqref{E:=(4.19)}, we  find
\begin{equation}\label{E:=(4.43)}
\begin{array}{lll}

K_3   K_5 K_4\left(1+\lambda^{-1}_n\right)^2\left\|\theta_{x}^n\right\|_{L^2(0,L_0)}^2\left\|u_{xx}^n\right\|_{L^2(0,L_0)}^2 \left\| u_{xx}\right\|_{L^2(0,L_0)}^4\nline

\leq 2m_1 \left\| u_{xx}^n\right\|_{L^2(0,L_0)}^8+\frac{K_3^4   K_5^4 K_4^4\left(1+\lambda^{-1}_n\right)^8\epsilon_{1,n}^4}{64 m_1^3}\,  \lambda^{-4\ell}_n.
\end{array}
\end{equation}
Also, taking $a=\epsilon_{8,n}\lambda^{-2\ell}_n$,  $b=\left\|u_{xx}^n\right\|_{L^2(0,L_0)}^2$, and $c=16\gamma^2\kappa^{-2}  K_5 \left\|\theta_{x}^n\right\|_{L^2(0,L_0)}^2$ in \eqref{E:=(A.5)}, then using \eqref{E:=(4.19)}, we  have
\begin{equation}\label{E:=(4.44)}
\begin{array}{lll}

K_3   K_5 \left\|\theta_{x}^n\right\|_{L^2(0,L_0)}^2\left\|u_{xx}^n\right\|_{L^2(0,L_0)}^2\epsilon_{8,n}\lambda^{-2\ell}_n

\nline

\leq m_1 \left\|u_{xx}^n\right\|_{L^2(0,L_0)}^8 +m_1 \epsilon_{8,n}^2\lambda^{-4\ell}_n+\frac{1024\gamma^8\kappa^{-8}  K_5^4 \left\|\theta_{x}^n\right\|_{L^2(0,L_0)}^8}{m_1^3}

\nline

\leq m_1 \left\| u_{xx}^n\right\|_{L^2(0,L_0)}^8+\left(m_1 \epsilon_{8,n}^2+\frac{1024\gamma^8\kappa^{-8}  K_5^4 \epsilon_{1,n}^4}{m_1^3}\right) \lambda^{-4\ell}_n.

\end{array}
\end{equation}
Substituting \eqref{E:=(4.42)}-\eqref{E:=(4.44)} in \eqref{E:=(4.41)}, we get \eqref{E:=(4.40)}, where
\begin{equation*}
\epsilon_{9,n}=K_5\epsilon_{7,n} \epsilon_{8,n}+\frac{K_5^2K_4^2\left(1+\lambda^{-1}_n\right)^4\epsilon_{7,n}^2}{4m_1}+\frac{K_3^4   K_5^4 K_4^4\left(1+\lambda^{-1}_n\right)^8\epsilon_{1,n}^4}{64 m_1^3}+m_1 \epsilon_{8,n}^2+\frac{1024\gamma^8\kappa^{-8}  K_5^4 \epsilon_{1,n}^4}{m_1^3}\to 0.
\end{equation*}

\end{proof}
\noindent \textbf{Proof of Lemma \ref{L:=4.5}.} First, from  \eqref{E:=(4.9)}, we have
\begin{equation*}
  u_{xx}^n=i\kappa\gamma^{-1} \lambda^{-1}_n \theta_{xx}^n+\gamma^{-1}\rho_0  \theta^n +i\gamma^{-1}\lambda_n^{-\ell-1}\left(\rho_0 h_3^n-\gamma\, (g_1^n)_{xx}\right).
\end{equation*}
Multiplying the above equation by $u_{xx}^n$ in $L^2(0,L_0)$, then taking the real parts and using integration by parts,  we get
\begin{equation*}
\begin{array}{lll}

\left\|u_{xx}^n\right\|_{L^2(0,L_0)}^2=-\kappa\gamma^{-1} \lambda^{-1}_n\Re\left\{i\int_0^{L_0} \theta_{x}^n\,  \overline{u_{xxx}^n}\,dx\right\}+\kappa\gamma^{-1} \lambda^{-1}_n\Re\left\{i\theta_{x}^n\left(L_0\right) \overline{u_{xx}^n}\left(L_0\right)-i\theta_{x}^n\left(0\right) \overline{u_{xx}^n}\left(0\right)\right\}

 \nline

 \hspace{2cm}+\gamma^{-1} \rho_0   \Re\left\{\int_0^{L_0}\theta^n  \overline{u_{xx}^n}\,dx\right\}+\gamma^{-1}\lambda^{-\ell-1}_n\Re\left\{i\int_0^{L_0}\left(\rho_0 h_3^n-\gamma\, (g_1^n)_{xx}\right) \overline{u_{xx}^n}\,dx\right\}.

 \end{array}
\end{equation*}
Consequently, using Cauchy Schwarz inequality, we  obtain
\begin{equation}\label{E:=(4.45)}
\begin{array}{lll}

\left\|u_{xx}^n\right\|_{L^2(0,L_0)}^2\leq \kappa\gamma^{-1} \lambda^{-1}_n\left\|\theta_{x}^n\right\|_{L^2(0,L_0)} \left\|u_{xxx}^n\right\|_{L^2(0,L_0)} +\kappa\gamma^{-1} \lambda^{-1}_n\left(\left|\theta_{x}^n\left(L_0\right)\right|+\left|\theta_{x}^n\left(0\right)\right| \right)\left| {u_{xx}^n}\right|_{\infty}

 \nline

 \hspace{2cm}+\gamma^{-1}\rho_0 \left\|\theta^n\right\|_{L^2(0,L_0)} \left\|u_{xx}^n\right\|_{L^2(0,L_0)}
+\gamma^{-1}\lambda^{-\ell-1}_n \left\|\rho_0h_3^n-\gamma\, (g_1^n)_{xx}\right\|_{L^2(0,L_0)} \left\|u_{xx}^n\right\|_{L^2(0,L_0)}.

 \end{array}
\end{equation}
Using the second estimation of \eqref{E:=(A.6)} in \eqref{E:=(4.45)},  one finds
\begin{equation}\label{E:=(4.46)}
\begin{array}{lll}

 \left\|u_{xx}^n\right\|_{L^2(0,L_0)}^8\leq 64\kappa^4\gamma^{-4} \lambda^{-4}_n\left\|\theta_{x}^n\right\|_{L^2(0,L_0)}^4 \left\|u_{xxx}^n\right\|_{L^2(0,L_0)}^4 +64\kappa^4\gamma^{-4} \lambda^{-4}_n\left(\left|\theta_{x}^n\left(L_0\right)\right|+\left|\theta_{x}^n\left(0\right)\right| \right)^4\left| {u_{xx}^n}\right|_{\infty}^4

 \nline

 \hspace{1.2cm}
+64\gamma^{-4} \left\|\theta^n\right\|_{L^2(0,L_0)}^4 \left\|u_{xx}^n\right\|_{L^2(0,L_0)}^4
 +64\gamma^{-4}\rho_0^4\lambda^{-4\ell-4}_n \left\|\rho_0h_3^n-\gamma\, (g_1^n)_{xx}\right\|_{L^2(0,L_0)}^4 \left\|u_{xx}^n\right\|_{L^2(0,L_0)}^4.

 \end{array}
\end{equation}
Taking $a=\left\|u_{xx}^n\right\|_{L^2(0,L_0)}^4$ and $b=64\gamma^{-4}\rho_0^4 \left\|\theta^n\right\|_{L^2(0,L_0)}^4$ in \eqref{E:=(A.4)}, then using \eqref{E:=(4.20)}, we get
\begin{equation}\label{E:=(4.47)}
64\gamma^{-4} \left\|\theta^n\right\|_{L^2(0,L_0)}^4 \left\|u_{xx}^n\right\|_{L^2(0,L_0)}^4\leq m_1\left\|u_{xx}^n\right\|_{L^2(0,L_0)}^8+\frac{1024\gamma^{-8}\rho_0^8\epsilon_{2,n}^4}{m_1}\lambda^{-4\ell}_n.
\end{equation}
On the other hand, we have
\begin{equation}\label{E:=(4.48)}
\left(\left|\theta_{x}^n\left(L_0\right)\right|+\left|\theta_{x}^n\left(0\right)\right| \right)^4\leq 4 \left(\left|\theta_{x}^n\left(L_0\right)\right|^2+\left|\theta_{x}^n\left(0\right)\right|^2 \right)^2.
\end{equation}
Substituting \eqref{E:=(4.47)} and  \eqref{E:=(4.48)} in \eqref{E:=(4.46)}, we  infer that
\begin{equation}\label{E:=(4.49)}
\begin{array}{lll}

 \left\|u_{xx}^n\right\|_{L^2(0,L_0)}^8\leq m_1\left\|u_{xx}^n\right\|_{L^2(0,L_0)}^8+ K_2 \lambda^{-4}_n\left\|\theta_{x}^n\right\|_{L^2(0,L_0)}^4 \left\|u_{xxx}^n\right\|_{L^2(0,L_0)}^4

\nline

 +K_5 \lambda^{-4}_n\left(\left|\theta_{x}^n\left(L_0\right)\right|^2+\left|\theta_{x}^n\left(0\right)\right|^2\right)^2 \left| u_{xx}^n\right|_{\infty}^4

 \nline

 +\left[\frac{1024\rho_0^8\gamma^{-8}\epsilon_{2,n}^4}{m_1}+64\gamma^{-4}\lambda^{-4}_n \left\|\rho_0h_3^n-\gamma\, (g_1^n)_{xx}\right\|_{L^2(0,L_0)}^4 \left\|u_{xx}^n\right\|_{L^2(0,L_0)}^4\right] \lambda^{-4\ell}_n.

 \end{array}
\end{equation}
Substituting \eqref{E:=(4.30)} and \eqref{E:=(4.40)} in  \eqref{E:=(4.49)}, we obtain
\begin{equation}\label{E:=(4.50)}
\left\|u_{xx}^n\right\|_{L^2(0,L_0)}^8\leq 6m_1\left\|u_{xx}^n\right\|_{L^2(0,L_0)}^8+\epsilon_{10,n}\lambda^{-4\ell}_n,
\end{equation}
where
\begin{equation*}
\epsilon_{10,n}=\epsilon_{6,n}+\epsilon_{9,n}+\frac{1024\gamma^{-8}\rho_0^8\epsilon_{2,n}^4}{m_1}+64\gamma^{-4}\lambda^{-4}_n \left\|\rho_0h_3^n-\gamma\, (g_1^n)_{xx}\right\|_{L^2(0,L_0)}^4 \left\|u_{xx}^n\right\|_{L^2(0,L_0)}^4.
\end{equation*}
Using \eqref{E:=(4.17)} and $h_3^n-\gamma\, (g_1^n)_{xx}\to 0$ in  $L^2(0,L_0)$, we get $\epsilon_{10,n}\to0.$ Thus, from \eqref{E:=(4.50)}, we get
\begin{equation*}
\left(1-6 m_1\right)\lambda^{4\ell}_n \left\|u_{xx}^n\right\|_{L^2(0,L_0)}^8 \leq \epsilon_{10,n}.
\end{equation*}
Taking $m_1=\frac{1}{12}$,  it holds that
 \begin{equation*}
 0\leq \frac{\lambda^{4\ell}_n}{2}\left\|u_{xx}^n\right\|_{L^2(0,L_0)}^8 \leq \epsilon_{10,n}\to 0.
 \end{equation*}
Thus, we obtain \eqref{E:=(4.21)}. Finally, substituting \eqref{E:=(4.21)} in    \eqref{E:=(4.37)},  we obtain \eqref{E:=(4.22)}. The proof is thus complete.\xqed{$\square$}
\begin{Lemma}\label{L:=4.11}
For all $\ell\geq0$, the solution $\left(\Phi_1^n,\Phi_2^n\right)\in D(\mathcal{A})$ of System \eqref{E:=(4.6)} satisfies the following asymptotic behavior  estimations
\begin{eqnarray}
\int_{0}^{L_0}\left|v^n\right|^2 dx=o\left(\lambda^{\ell_1}_n\right),\quad \int_{0}^{L_0}\left|v^n_x\right|^2 dx=o\left(\lambda^{\ell_1}_n\right),\label{EQ=(4.51)}
\nline
\int_{L_0}^L\left|y_{xx}^n\right|^2 dx=o\left(\lambda^{\ell_1}_n\right),\quad \int_{L_0}^L\left|z^n\right|^2 dx=o\left(\lambda^{\ell_1}_n\right),\quad  \int_{L_0}^L\left|z_x^n\right|^2 dx=o\left(\lambda^{\ell_1}_n\right),\label{E:=(4.52)}
\end{eqnarray}
where
\begin{equation*}
\ell_1=\left\{
\begin{array}{lll}

1-\ell,&\text{if } \alpha_1\geq \alpha_2,

\nline

2-\ell,&\text{if } \alpha_1<\alpha_2.

\end{array}\right.
\end{equation*}
\end{Lemma}
\begin{proof} The proof will be split into several steps:\\[0.1in]
\textbf{Step 1.} In this step, we prove the following asymptotic behavior estimates:
\begin{equation}\label{E:=(4.53)}
\int_0^{L_0}\left|u_{x}^n\right|^2 dx=o\left(\lambda^{-\ell}_n\right),\quad  \int_0^{L_0}\left|u^n\right|^2 dx=o\left(\lambda^{-\ell}_n\right),\quad \left|u_x^n\right|_{\infty}=o\left(\lambda^{-\frac{\ell}{2}}_n\right),\quad \left|u^n\right|_{\infty}=o\left(\lambda^{-\frac{\ell}{2}}_n\right).
\end{equation}
 In fact, since $u\in H^3(0,L_0)$ with $u^n(0)=u_x^n(0)=0,$ then using \eqref{E:=(4.21)}, Poincaré's inequality and trace theorem, we get \eqref{E:=(4.53)}.\\[0.1in]
\textbf{Step 2.} In this step, we prove the following asymptotic behavior estimate:
\begin{equation}\label{E:=(4.54)}
\begin{array}{lll}

\lambda^2_n \int_0^{L_0}\left(\alpha_1 \left|u_{x}^n\right|^2 + \rho_1\left|u^n\right|^2\right)  dx
-\beta_1\int_0^{L_0} \left|u_{xx}^n\right|^2 dx

\nline

+ \lambda^2_n \int_{L_0}^{L}\left(\alpha_2 \left|y_{x}^n\right|^2 +\rho_2 \left|y^n\right|^2\right)  dx
-\beta_2\int_{L_0}^{L} \left|y_{xx}^n\right|^2 dx=o\left(\lambda^{-\ell}_n\right).

\end{array}
\end{equation}
For this aim, first, multiplying \eqref{E:=(4.7)} by $-u^n$ in $L^2(0,L_0)$, then taking the real parts, using integration by parts and the fact that $u^n(0)=u_x^n(0)=0$,  we  obtain
\begin{equation}\label{E:=(4.55)}
\begin{array}{llll}

\lambda^2_n \int_0^{L_0}\left(\alpha_1 \left|u_{x}^n\right|^2 +\rho_1 \left|u^n\right|^2\right)  dx
-\beta_1\int_0^{L_0} \left|u_{xx}^n\right|^2 dx +\beta_1\Re\left\{u_{xx}^n(L_0) \overline{u}_x^n(L_0) \right\}

\nline

-\Re\left\{\left(\beta_1\, u_{xxx}^n(L_0)+\alpha_1\,\lambda^{-\ell}_n   (h_1^n)_x(L_0)+\lambda^2_n \alpha_1\, u_x^n(L_0)+\gamma  \theta_{x}^n(L_0)+i \lambda^{-\ell+1}_n (g_1^n)_x(L_0) \right)\overline{u}^n(L_0)\right\}

\nline

=-\rho_1\lambda^{-\ell}_n\Re\left\{\int_0^{L}\left(\lambda^{-1}_n h_1^n+ig_1^n\right)\, \lambda_n \overline{u}\, dx\right\}-\alpha_1\,\lambda^{-\ell}_n\Re\left\{\int_0^{L_0}   (h_1^n)_x \overline{u}_x^n\, dx\right\}

\nline

+\lambda^{-\ell}_n\Re\left\{i\int_0^{L}  (g_1^n)_{x}\, \lambda_n \overline{u}_x^n\, dx\right\}-\gamma  \Re\left\{\int_0^{L_0}\theta_{x}^n\overline{u}_x^n\, dx\right\}.

\end{array}
\end{equation}
Next, multiplying \eqref{E:=(4.8)} by $-y^n$ in $L^2(L_0,L)$, then taking the real parts, using integration by parts and the fact that $y^n(L)=y_x^n(L)=0$,  one derives
\begin{equation}\label{E:=(4.56)}
\begin{array}{llll}

\lambda^2_n \int_{L_0}^{L}\left(\alpha_2 \left|y_{x}^n\right|^2 +\rho_2 \left|y^n\right|^2\right)  dx
-\beta_2\int_{L_0}^{L} \left|y_{xx}^n\right|^2 dx -\beta_2\Re\left\{y_{xx}^n(L_0) \overline{y}_x^n(L_0) \right\}

\nline

+\Re\left\{\left(\beta_2\, y_{xxx}^n(L_0)+\alpha_2\,\lambda^{-\ell}_n   (h_2^n)_x(L_0)+\lambda^2_n \alpha_2\, y_x^n(L_0)+i \lambda^{-\ell+1}_n (g_2^n)_x(L_0) \right)\overline{y^n}(L_0)\right\}

\nline

=-\rho_2\lambda^{-\ell}_n\Re\left\{\int_{L_0}^{L}\left(\lambda^{-1}_n h_2^n+ig_2^n\right)\, \lambda_n \overline{y^n}\, dx\right\}-\alpha_2\,\lambda^{-\ell}\Re\left\{\int_{L_0}^{L}   (h_2^n)_x \overline{y^n}_x\, dx\right\}

\nline

+\lambda^{-\ell}_n\Re\left\{i\int_{L_0}^{L}  (g_2^n)_{x}\, \lambda_n \overline{y^n}_x\, dx\right\}.

\end{array}
\end{equation}
Adding \eqref{E:=(4.55)} and \eqref{E:=(4.56)}, then using \eqref{E:=(3.9)}, \eqref{E:=(4.11)}, and the fact that $u^n(L_0)=y^n(L_0),\,  u_x^n(L_0)=y_x^n(L_0),$ we  find
\begin{equation}\label{E:=(4.57)}
\begin{array}{llll}

\lambda^2_n \int_0^{L_0}\left(\alpha_1 \left|u_{x}^n\right|^2 +\rho_1 \left|u^n\right|^2\right)  dx
-\beta_1\int_0^{L_0} \left|u_{xx}^n\right|^2 dx+ \lambda^2_n \int_{L_0}^{L}\left(\alpha_2 \left|y_{x}^n\right|^2 + \rho_2\left|y^n\right|^2\right)  dx

\nline

 -\beta_2\int_{L_0}^{L} \left|y_{xx}^n\right|^2 dx
=-\rho_1\lambda^{-\ell}_n\Re\left\{\int_0^{L}\left(\lambda^{-1}_n h_1^n+ig_1^n\right)\, \lambda_n \overline{u^n}\, dx\right\}-\alpha_1\,\lambda^{-\ell}_n\Re\left\{\int_0^{L_0}   (h_1^n)_x \overline{u^n_x}\, dx\right\}

\nline

+\lambda^{-\ell}_n\Re\left\{i\int_0^{L}  (g_1^n)_{x}\, \lambda_n \overline{u^n_x}\, dx\right\}
-\rho_2\lambda^{-\ell}_n\Re\left\{\int_{L_0}^{L}\left(\lambda^{-1}_n h_2^n+ig_2^n\right)\, \lambda_n \overline{y^n}\, dx\right\}

\nline

-\alpha_2\,\lambda^{-\ell}_n\Re\left\{\int_{L_0}^{L}   (h_2^n)_x \overline{y^n_x}\, dx\right\}
+\lambda^{-\ell}_n\Re\left\{i\int_{L_0}^{L}  (g_2^n)_{x}\, \lambda_n \overline{y^n_x}\, dx\right\}
-\gamma  \Re\left\{\int_0^{L_0}\theta_{x}^n\overline{u^n_x}\, dx\right\}.

\end{array}
\end{equation}
Using \eqref{E:=(4.17)}, the fact that $h^n=\left(h_1^n,h_2^n,h_3^n\right)\to 0$  in $\mathbb{V}_2$ and $g^n=\left(g_1^n,g_2^n\right)\to 0$ in $\mathbb{W}_1$,  it holds that
\begin{equation}\label{E:=(4.58)}
\begin{array}{llll}
-\rho_1\lambda^{-\ell}_n\Re\left\{\int_0^{L}\left(\lambda^{-1}_n h_1^n+ig_1^n\right)\, \lambda_n \overline{u^n}\, dx\right\}-\alpha_1\,\lambda^{-\ell}_n\Re\left\{\int_0^{L_0}   (h_1^n)_x \overline{u^n_x}\, dx\right\}

\nline

+\lambda^{-\ell}_n\Re\left\{i\int_0^{L}  (g_1^n)_{x}\, \lambda_n \overline{u^n_x}\, dx\right\}
-\rho_2\lambda^{-\ell}_n\Re\left\{\int_{L_0}^{L}\left(\lambda^{-1}_n h_2^n+ig_2^n\right)\, \lambda_n \overline{y^n}\, dx\right\}

\nline

-\alpha_2\,\lambda^{-\ell}_n\Re\left\{\int_{L_0}^{L}   (h_2^n)_x \overline{y^n_x}\, dx\right\}

+\lambda^{-\ell}_n\Re\left\{i\int_{L_0}^{L}  (g_2^n)_{x}\, \lambda_n \overline{y^n_x}\, dx\right\}=o\left(\lambda^{-\ell}_n\right).

\end{array}
\end{equation}
On the other hand, from \eqref{E:=(4.18)} and \eqref{E:=(4.53)}, we  infer
\begin{equation}\label{E:=(4.59)}
-\gamma  \Re\left\{\int_0^{L_0}\theta_{x}^n\overline{u^n_x}\, dx\right\}= o\left(\lambda^{-\ell}_n\right).
\end{equation}
Substituting \eqref{E:=(4.58)} and \eqref{E:=(4.59)}  in \eqref{E:=(4.57)}, we obtain \eqref{E:=(4.54)}. \\[0.1in]
\textbf{Step 3.} In this step, we prove the following asymptotic behavior estimate:
\begin{equation}\label{E:=(4.60)}
\begin{array}{lllll}

  \lambda^2_n \int_0^{L_0}\left(-\alpha_1\left|u_x^n\right|^2+\rho_1 \left|u^n\right|^2       \right)  dx+ 3\beta_1 \int_0^{L_0} \left|u_{xx}^n\right|^2  dx+ \lambda^2_n\int_{L_0}^L\left( -\alpha_2\left|y_x^n\right|^2 + \rho_2  \left|y^n\right|^2   \right)  dx

 \nline

  +3\beta_2 \int_{L_0}^L  \left|y_{xx}^n\right|^2 dx
 +\lambda^2_n \left(L-L_0\right)\left[\left(\alpha_1-\alpha_2\right)\left|u_x^n(L_0)\right|^2 +\left(\rho_1-\rho_2\right)\left|u^n(L_0)\right|^2\right] =o\left(\lambda^{1-\ell}_n\right).

\end{array}
\end{equation}
For this aim, first, multiplying \eqref{E:=(4.7)} by $2\left(x-L\right) u_{x}$ in $L^2(0,L_0)$, taking the real parts, then using integration by parts and the fact that $u^n(0)=u_x^n(0)=0$,  we get
\begin{equation}\label{E:=(4.61)}
\begin{array}{lllll}

\lambda^2_n \int_0^{L_0}\left(-\alpha_1\left|u_x^n\right|^2+\rho_1 \left|u^n\right|^2\right)  dx+ 3\beta_1 \int_0^{L_0} \left|u_{xx}^n\right|^2  dx

\nline

-2\left(L-L_0\right)\Re\bigg\{\big(\beta_1\, u_{xxx}^n(L_0)+\alpha_1\,\lambda^{-\ell}_n  \left(  (h_1^n)_x(L_0)+i\lambda_n (g_1^n)_x(L_0)\right)

\nline

\hspace{2.7cm}+\alpha_1\lambda^2_n  u_x(L_0)+\gamma \theta_{x}^n(L_0)\big)\overline{u_x^n}(L_0)\bigg\}
-2\beta_1\Re\left\{ u_{xx}^n(L_0) \overline{u_x^n}(L_0)\right\}

\nline

+\beta_1\left(L-L_0\right) \left|u_{xx}^n(L_0)\right|^2-\beta_1\, L \left|u_{xx}^n(0)\right|^2

\nline

+\lambda^2_n\left(L-L_0\right)\left(\alpha_1\left|u_x^n(L_0)\right|^2 +\rho_1 \left|u^n(L_0)\right|^2\right)=2\rho_1\lambda^{-\ell}_n\Re\left\{\int_0^{L_0}\left(x-L\right) \left(\lambda^{-1}_n h_1^n+i g_1^n\right)  \lambda_n\overline{u_x^n}\, dx\right\}

\nline

+2\alpha_1 \lambda^{-\ell}_n\Re\left\{ \int_0^{L_0}   (h_1^n)_x\left(\left(x-L\right)\overline{u_x^n}\right)_x\, dx\right\}+2\alpha_1\lambda^{1-\ell}_n\Re\left\{i \int_0^{L_0} (g_1^n)_x\left(\left(x-L\right) \overline{u_x^n}\right)_x\, dx\right\}

\nline

+2\gamma \Re\left\{\int_0^{L_0} \theta_{x}^n\left(\left(x-L\right) \overline{u_x^n}\right)_x\, dx\right\} .

\end{array}
\end{equation}
Next, multiplying \eqref{E:=(4.8)} by $2\left(x-L\right) y_{x}^n$ in $L^2(L_0,L)$, taking the real parts, then using integration by parts and the fact that $y^n(L)=y_x^n(L)=0$,  we obtain
\begin{equation}\label{E:=(4.62)}
\begin{array}{lllll}

 \lambda^2_n\int_{L_0}^L\left( -\alpha_2\left|y_x^n\right|^2 +   \left|y^n\right|^2   \right)  dx+3\beta_2 \int_{L_0}^L  \left|y_{xx}^n\right|^2 dx

\nline

+2\left(L-L_0\right)\Re\left\{\left(\beta_2\, y_{xxx}^n(L_0)+\alpha_2\,\lambda^{-\ell}_n \left(  (h_2)_x^n(L_0)+i\lambda_n (g_2^n)_x(L_0)\right)+\alpha_2\lambda^2_n  y_x^n(L_0)\right)\overline{y_x^n}(L_0)\right\}

\nline

  +2\beta_2 \Re\left\{y_{xx}^n(L_0) \overline{y_x^n}(L_0)\right\}-\beta_2\left(L-L_0\right) \left|y_{xx}^n(L_0)\right|^2 -\lambda^2_n \left(L-L_0\right)\left(\alpha_2\left|y_x^n(L_0)\right|^2 +\rho_2 \left|y^n(L_0)\right|^2\right)

\nline

=2\rho_2\lambda^{-\ell}_n\Re\left\{\int_{L_0}^L\left(x-L\right)\left(\lambda^{-1}_n h_2^n+i g_2^n\right)  \lambda_n\overline{y_x^n}\, dx\right\}+2\alpha_2 \lambda^{-\ell}_n\Re\left\{ \int_{L_0}^L   (h_2^n)_x\left(\left(x-L\right) \overline{y_x^n}\right)_x\, dx\right\}

\nline

+2\alpha_2\lambda^{1-\ell}_n\Re\left\{i \int_{L_0}^L (g_2^n)_x\left(\left(x-L\right) \overline{y_x^n}\right)_x\, dx\right\}.

\end{array}
\end{equation}
Adding \eqref{E:=(4.61)} and \eqref{E:=(4.62)}, then using \eqref{E:=(3.9)}, \eqref{E:=(4.11)}, and the fact that $u^n(L_0)=y^n(L_0),\,  u_x^n(L_0)=y_x^n(L_0),$ we get
\begin{equation*}
\begin{array}{lllll}

  \lambda^2_n \int_0^{L_0}\left(-\alpha_1\left|u_x^n\right|^2+\rho_1 \left|u^n\right|^2       \right)  dx+ 3\beta_1 \int_0^{L_0} \left|u_{xx}^n\right|^2  dx+ \lambda^2_n\int_{L_0}^L\left( -\alpha_2\left|y_x^n\right|^2 + \rho_2  \left|y^n\right|^2   \right)  dx

  \nline

  +3\beta_2 \int_{L_0}^L  \left|y_{xx}^n\right|^2 dx+\lambda^2_n \left(L-L_0\right)\left(\alpha_1-\alpha_2\right)\left|u_x^n(L_0)\right|^2 +\lambda^2_n \left(L-L_0\right)\left(\rho_1-\rho_2\right)\left|u^n(L_0)\right|^2

 \nline

 =\frac{\beta_1\left(\beta_1-\beta_2\right)}{\beta_2}\left(L-L_0\right) \left|u_{xx}^n(L_0)\right|^2 +\beta_1\, L \left|u_{xx}^n(0)\right|^2

\nline

+2\rho_1\lambda^{-\ell}_n\Re\left\{\int_0^{L_0}\left(x-L\right) \left(\lambda^{-1}_n h_1^n+i g_1^n\right)  \lambda_n\overline{u_x^n}\, dx\right\}

\end{array}
\end{equation*}
\begin{equation}\label{E:=(4.63)}
\begin{array}{lllll}

+2\rho_2\lambda^{-\ell}_n\Re\left\{\int_{L_0}^L\left(x-L\right)\left(\lambda^{-1}_n h_2^n+i g_2^n\right)  \lambda_n\overline{y_x^n}\, dx\right\}

\nline

+2\alpha_1\lambda^{1-\ell}_n\Re\left\{i \int_0^{L_0} (g_1^n)_x\left(\left(x-L\right) \overline{u_x^n}\right)_x\, dx\right\}+2\alpha_2\lambda^{1-\ell}_n\Re\left\{i \int_{L_0}^L (g_2^n)_x\left(\left(x-L\right) \overline{y_x^n}\right)_x\, dx\right\}

\nline

+2\alpha_1 \lambda^{-\ell}_n\Re\left\{ \int_0^{L_0}   (h_1^n)_x\left(\left(x-L\right)\overline{u_x^n}\right)_x\, dx\right\}+2\alpha_2 \lambda^{-\ell}_n\Re\left\{ \int_{L_0}^L   (h_2^n)_x\left(\left(x-L\right) \overline{y_x^n}\right)_x\, dx\right\}

\nline

+2\gamma \Re\left\{\int_0^{L_0} \theta_{x}^n\left( \overline{u_x^n}+\left(x-L\right) \overline{u_{xx}^n}\right)\, dx\right\}.

\end{array}
\end{equation}
From \eqref{E:=(4.22)}, it holds that
\begin{equation}\label{E:=(4.64)}
\frac{\beta_1\left(\beta_1-\beta_2\right)}{\beta_2}\left(L-L_0\right) \left|u_{xx}^n(L_0)\right|^2 +\beta_1\, L \left|u_{xx}^n(0)\right|^2=o\left(\lambda^{1-\ell}_n\right).
\end{equation}
Using \eqref{E:=(4.17)}, the fact that  $h^n=\left(h_1^n,h_2^n,h_3^n\right)\to 0$  in $\mathbb{V}_2$ and $g^n=\left(g_1^n,g_2^n\right)\to 0$ in $\mathbb{W}_1$, we derive
\begin{equation}\label{E:=(4.65)}
\begin{array}{lllll}

+2\rho_1\lambda^{-\ell}_n\Re\left\{\int_0^{L_0}\left(x-L\right) \left(\lambda^{-1}_n h_1^n+i g_1^n\right)  \lambda_n\overline{u_x^n}\, dx\right\}+2\alpha_1\lambda^{1-\ell}_n\Re\left\{i \int_0^{L_0} (g_1^n)_x\left(\left(x-L\right) \overline{u_x^n}\right)_x\, dx\right\}

\nline

+2\rho_2\lambda^{-\ell}_n\Re\left\{\int_{L_0}^L\left(x-L\right)\left(\lambda^{-1}_n h_2^n+i g_2^n\right)  \lambda_n\overline{y_x^n}\, dx\right\}
+2\alpha_2\lambda^{1-\ell}_n\Re\left\{i \int_{L_0}^L (g_2^n)_x\left(\left(x-L\right) \overline{y_x^n}\right)_x\, dx\right\}

\nline

+2\alpha_1 \lambda^{-\ell}_n\Re\left\{ \int_0^{L_0}   (h_1^n)_x\left(\left(x-L\right)\overline{u_x^n}\right)_x\, dx\right\}+2\alpha_2 \lambda^{-\ell}_n\Re\left\{ \int_{L_0}^L   (h_2^n)_x\left(\left(x-L\right) \overline{y_x^n}\right)_x\, dx\right\}=o\left(\lambda^{1-\ell}_n\right).

\end{array}
\end{equation}
On the other hand, using Cauchy–Schwarz inequality, \eqref{E:=(4.18)},  \eqref{E:=(4.21)} and \eqref{E:=(4.53)}, we obtain
\begin{equation}\label{E:=(4.66)}
2\gamma \Re\left\{\int_0^{L_0} \theta_{x}^n\left( \overline{u_x^n}+\left(x-L\right) \overline{u_{xx}^n}\right)\, dx\right\}=o\left(\lambda^{-\ell}_n\right).
\end{equation}
Substituting \eqref{E:=(4.64)}-\eqref{E:=(4.66)} in \eqref{E:=(4.63)}, we get \eqref{E:=(4.60)}.\\[0.1in]
\textbf{Step 4.}  In this step, we will prove \eqref{EQ=(4.51)}-\eqref{E:=(4.52)}. First, adding \eqref{E:=(4.54)} and \eqref{E:=(4.60)}, we find that
\begin{equation}\label{E:=(4.67)}
\begin{array}{lllll}

  2\rho_1\lambda^2_n \int_0^{L_0} \left|u^n\right|^2        dx+ 2\beta_1 \int_0^{L_0} \left|u_{xx}^n\right|^2  dx+ 2\rho_2\lambda^2_n\int_{L_0}^L   \left|y^n\right|^2   dx+2\beta_2 \int_{L_0}^L  \left|y_{xx}^n\right|^2 dx

\nline

+\lambda^2_n \left(L-L_0\right)\left[\left(\alpha_1-\alpha_2\right)\left|u_x^n(L_0)\right|^2 +\left(\rho_1-\rho_2\right)\left|u^n(L_0)\right|^2\right] =o\left(\lambda^{1-\ell}_n\right).

\end{array}
\end{equation}
We distinguish two cases:\\[0.1in]
\textbf{Case 1.} If $\alpha_1\geq \alpha_2$ and $\rho_1\geq \rho_2$, then  substituting \eqref{E:=(4.21)} in \eqref{E:=(4.67)}, one derives
\begin{equation}\label{E:=(4.68)}
\left\{
\begin{array}{lllll}

  \lambda^2_n \int_0^{L_0} \left|u^n\right|^2dx=o\left(\lambda^{1-\ell}_n\right),\quad  \lambda^2_n\int_{L_0}^L   \left|y^n\right|^2   dx=o\left(\lambda^{1-\ell}_n\right),

\nline

\int_{L_0}^L  \left|y_{xx}^n\right|^2 dx=o\left(\lambda^{1-\ell}_n\right), \quad \lambda^2_n\left|u_x^n(L_0)\right|^2 =o\left(\lambda^{1-\ell}_n\right),\quad \lambda^2_n\left|u^n(L_0)\right|^2 =o\left(\lambda^{1-\ell}_n\right).

\end{array}
\right.
\end{equation}
Substituting \eqref{E:=(4.21)} and \eqref{E:=(4.68)} in \eqref{E:=(4.54)}, we get
\begin{equation*}
 \alpha_1\lambda^2_n \int_0^{L_0} \left|u_{x}^n\right|^2  dx +
 \alpha_2\lambda^2_n \int_{L_0}^{L} \left|y_{x}^n\right|^2  dx =o\left(\lambda^{1-\ell}_n\right).
\end{equation*}
Consequently, we obtain
\begin{equation}\label{E:=(4.69)}
\lambda^2_n \int_0^{L_0} \left|u_{x}^n\right|^2  dx=o\left(\lambda^{1-\ell}_n\right)\quad \text{and}\quad
 \lambda^2_n \int_{L_0}^{L} \left|y_{x}^n\right|^2  dx =o\left(\lambda^{1-\ell}_n\right).
\end{equation}
Thus\textcolor{red}{,} from \eqref{E:=(4.13)}, \eqref{E:=(4.68)}, \eqref{E:=(4.69)}, and the fact that $g=\left(g_1,g_2\right)\to 0$ in $\mathbb{W}_1$, we get \eqref{EQ=(4.51)}-\eqref{E:=(4.52)}.\\[0.1in]
\textbf{Case 2.} If $\alpha_1< \alpha_2$ or $\rho_1< \rho_2$, then substituting \eqref{E:=(4.21)} and \eqref{E:=(4.53)} in \eqref{E:=(4.67)}, we infer
\begin{equation*}
\lambda^2_n\int_{L_0}^L   \left|y^n\right|^2   dx+2\beta_2 \int_{L_0}^L  \left|y_{xx}^n\right|^2 dx=o\left(\lambda^{2-\ell}_n\right).
\end{equation*}
Consequently, we obtain
\begin{equation}\label{E:=(4.70)}
\lambda^2_n\int_{L_0}^L   \left|y^n\right|^2   dx=o\left(\lambda^{2-\ell}_n\right)\quad\text{and}\quad \int_{L_0}^L  \left|y_{xx}^n\right|^2 dx=o\left(\lambda^{2-\ell}_n\right).
\end{equation}
Inserting \eqref{E:=(4.21)}, \eqref{E:=(4.70)}, and \eqref{E:=(4.53)} in \eqref{E:=(4.54)}, we get
\begin{equation}\label{E:=(4.71)}
\lambda^2_n\int_{L_0}^L   \left|y_x^n\right|^2   dx=o\left(\lambda^{2-\ell}_n\right).
\end{equation}
Thus from \eqref{E:=(4.13)}, \eqref{E:=(4.53)}, \eqref{E:=(4.70)},  \eqref{E:=(4.71)}, and the fact that $g^n=\left(g_1^n,g_2^n\right)\to 0$ in $\mathbb{W}_1$, we get \eqref{EQ=(4.51)}-\eqref{E:=(4.52)}. The proof is thus complete.
\end{proof}
\noindent \textbf{Proof of Theorem \ref{T:=4.1}.} When $\alpha_1\geq \alpha_2$ and $\rho_1\geq \rho_2$, we choose $\ell=1$, then from Lemmas \ref{L:=4.4},   \ref{L:=4.5}, and \ref{L:=4.11}, we get $\left\|\Phi^n\right\|_{{\mathcal{H}}} = o(1)$ which contradicts \eqref{E:=(4.4)}.  This implies that
\begin{equation*}
\sup_{\lambda\in\mathbb{R}}\frac{1}{|\lambda|}\left\|\left(i\lambda I-\mathcal{A}\right)^{-1}\right\|_{\mathcal{L}\left(\mathcal{H}\right)}<+\infty.
\end{equation*}
The result follows from  Theorem \ref{T:=A.4}.  \xqed{$\square$}\\[0.1in]
\noindent \textbf{Proof of Theorem \ref{T:=4.2}.}  When $\alpha_1< \alpha_2$ or $\rho_1< \rho_2$, we choose $\ell=2$, then  from Lemmas \ref{L:=4.4},   \ref{L:=4.5}, and \ref{L:=4.11}, we get $\left\|\Phi^n\right\|_{{\mathcal{H}}} = o(1)$ which contradicts \eqref{E:=(4.4)}. Hence
\begin{equation*}
\sup_{\lambda\in\mathbb{R}}\frac{1}{\lambda^2}\left\|\left(i\lambda I-\mathcal{A}\right)^{-1}\right\|_{\mathcal{L}\left(\mathcal{H}\right)}<+\infty.
\end{equation*}
The result follows from  Theorem \ref{T:=A.4}.  \xqed{$\square$}
\section*{Conclusion and open problems}
\subsection{Conclusion} In this paper, we investigate the stability of a transmission Rayleigh beam with heat conduction.  A polynomial energy decay rate has been obtained which  depends on the physical constant. We obtain the following result:\\
 $\bullet$ A polynomial energy decay rate of type $t^{-2}$ if $\rho_1\geq \rho_2$ and $\alpha_1\geq \alpha_2$. \\
 $\bullet$ A polynomial energy decay rate of type $t^{-1}$ if $\rho_1<\rho_2$ or $\alpha_1<\alpha_2$.
 \subsection{Open Problems} In this part, we present some open problems:
 \begin{enumerate}
 \item[${\rm \textbf{(OP1)}}$] The optimality of the polynomial decay rate of the System \eqref{E:=(1.1)}-\eqref{E:=(1.12)}. But, we conjecture that the polynomial energy decay rate obtained in Theorem \ref{T:=4.1} and Theorem \ref{T:=4.2} is optimal. The idea of the proof is to find a sequence $(\lambda_n)_n\subset \mathbb{R}^{\ast}_+$ with $|\lambda_n|\to +\infty$ and a sequence of vectors $(U_n)_n\subseteq D(\mathcal{A})$ such that $(i\lambda_n-\mathcal{A})U_n=F_n$ is bounded in $\mathcal{H}$ and
 $$
 \lim_{n\to +\infty}\lambda_n^{-2+\varepsilon}\|U_n\|_{\mathcal{H}}=\infty.
 $$
 (see for example Theorem 3.1 in \cite{MR4215149}, Theorem 5.1 in \cite{MR4258404} and \cite{MR4213670}). Depending on the boundary conditions and the transmission conditions, this approach and the construction of the vector $(U_n)$ are not feasible and the question is still an open problem.
  \item[${\rm \textbf{(OP2)}}$] What happened if we consider a heat conduction with memory, where the hereditary heat conduction is due to Coleman-Gurtin law or Gurtin-Pipkin law? (See for instance \cite{MR4500789,MR4213670,MR3263148,MR4491743})
 \end{enumerate}

\appendix
\section{Notions of stability and theorems used}\label{S:=A}
\noindent We introduce here the notions of stability that we encounter in this work.\begin{Definition}\label{D:=A.1}
Assume that $A$ is the generator of a C$_0$-semigroup of contractions $\left(e^{tA}\right)_{t\geq0}$  on a Hilbert space  $H$. The  $C_0$-semigroup $\left(e^{tA}\right)_{t\geq0}$  is said to be
\begin{enumerate}
\item[1.]  strongly stable if
\begin{equation*}
\lim_{t\to +\infty} \|e^{tA}x_0\|_{\mathcal{H}}=0, \quad\forall \ x_0\in H;
\end{equation*}
\item[3.] polynomially stable if there exist two positive constants $C$ and $\alpha$ such that
\begin{equation*}
 \|e^{tA}x_0\|_{H}\leq C t^{-\alpha}\|Ax_0\|_{H},  \quad\forall\
t>0,  \ \forall \ x_0\in D\left(A\right).
\end{equation*}
\xqed{$\square$}
\end{enumerate}
\end{Definition}
\noindent We now look  for necessary conditions to show the strong stability of the $C_0$-semigroup $\left(e^{t\mathcal{A}}\right)_{t\geq0}$. We will rely on the following result obtained by Arendt and Batty in \cite{Arendt-Batty-1988}.
\begin{Theorem}[Arendt and Batty in \cite{Arendt-Batty-1988}]\label{T:=A.2}
Assume that $A$ is the generator of a C$_0-$semigroup of contractions $\left(e^{tA}\right)_{t\geq0}$  on a Hilbert space $H$. If
 \begin{enumerate}
 \item[1.]  $A$ has no pure imaginary eigenvalues,\qquad 2.  $\sigma\left(A\right)\cap i\mathbb{R}$ is countable,
 \end{enumerate}
where $\sigma\left(A\right)$ denotes the spectrum of $A$, then the $C_0$-semigroup $\left(e^{t\mathcal{A}}\right)_{t\geq0}$  is strongly stable.\xqed{$\square$}
\end{Theorem}
 \begin{Corollary}\label{C:=A.3}
If the resolvent $(I-A)^{-1}$ of $A$ is compact, then the spectrum of $A$ only consists of eigenvalues of $A$ (see Theorem 6.29 in \cite{Kato-1995}). Thus, the state of Theorem \ref{T:=A.2} lessens to $A$ has no pure imaginary eigenvalues.\xqed{$\square$}
\end{Corollary}
\noindent For necessary conditions to show the polynomial stability of the $C_0$-semigroup $\left(e^{tA}\right)_{t\geq0}$\textcolor{red}{,} we will rely on the frequency domain approach method  that has been obtained by Batty in \cite{Batty-Duyckaerts-2008}, Borichev and  Tomilov in \cite{Borichev-Tomilov-2010},  Liu and  Rao in \cite{Liu-Rao-2005}.
\begin{Theorem}[Batty in \cite{Batty-Duyckaerts-2008}, Borichev and  Tomilov in \cite{Borichev-Tomilov-2010},  Liu and  Rao in \cite{Liu-Rao-2005}]\label{T:=A.4}
Assume that $A$ is the generator of a strongly continuous semigroup of contractions $\left(e^{tA}\right)_{t\geq0}$  on $H$.   If   $\sigma\left(A\right)\cap\ i\mathbb{R}=\emptyset$, then for a fixed $\ell>0$ the following conditions are equivalent
\begin{enumerate}
\item[1.] $\sup_{\lambda\in\mathbb{R}}\frac{1}{|\lambda|^{\ell}}\left\|\left(i\lambda I-A\right)^{-1}\right\|_{\mathcal{L}\left(H\right)}<+\infty,$
\item[2.] $\displaystyle{\|e^{tA}x_0\|_{H} \leq \dfrac{C}{t^{{\frac{1}{\ell}}}} \ \|x_0\|_{D\left(\mathcal{A}\right)} \quad\forall\ t>0,\ x_0\in D\left(A\right)},\ $ for some $C>0.$\xqed{$\square$}
\end{enumerate}
\end{Theorem}
\noindent We will recall two forms of Gagliardo-Nirenberg inequality (see \cite{Liu-Zheng-1999}) which will be used in this work.
\begin{Theorem}\label{T:=A.5}
$\\ \vspace{-0.5cm}$
\begin{enumerate}
\item[1.]  There are two positive constants $K_6$ and $K_7$ such that, for any $\zeta\in H^2(0,L_0)$, we have
\begin{equation}\label{E:=(A.1)}
\left\|\zeta_x\right\|_{L^2(0,L_0)}\leq K_6 \left\|\zeta_{xx}\right\|^{\frac{1}{2}}_{L^2(0,L_0)} \left\|\zeta\right\|^{\frac{1}{2}}_{L^2(0,L_0)} +K_7 \left\|\zeta\right\|_{L^2(0,L_0)}.
\end{equation}
\item[2.] There are two positive constants $K_8$ and $K_9$ such that, for any $\psi\in H^1(0,L_0)$, we have
\begin{equation}\label{E:=(A.2)}
\left|\psi\right|_{\infty}\leq K_8 \left\|\psi_{x}\right\|^{\frac{1}{2}}_{L^2(0,L_0)} \left\|\psi\right\|^{\frac{1}{2}}_{L^2(0,L_0)} +K_9 \left\|\psi\right\|_{L^2(0,L_0)}.
\end{equation}
\xqed{$\square$}
\end{enumerate}
\end{Theorem}
\noindent We will recall Young inequality and we will prove some inequalities that will be used in this work.
\begin{Lemma} \label{L:=A.6} $\\ \vspace{-0.5cm}$
\begin{enumerate}
\item[1.] For all positive numbers $p$ and $q$, the Young inequality is given by
\begin{equation}\label{E:=(A.3)}
p\, q\leq \frac{p^2}{2}+\frac{q^2}{2}.
\end{equation}
\item[2.]  For all $0<m_1<1$ and for all positive numbers $a$ and $b$,   we have
\begin{equation}\label{E:=(A.4)}
ab\leq m_1 a^2+\frac{b^2}{4m_1}.
\end{equation}
\item[3.] For all $0<m_1<1$ and for all positive numbers $a,\ b,$ and $c$,  we have
\begin{equation}\label{E:=(A.5)}
a\, b\, c\leq \left(a^2+b^4\right) m_1+\frac{c^4}{64m_1^3}.
\end{equation}
\end{enumerate}
\end{Lemma}
\begin{proof} Firstly, taking $p=\sqrt{2m_1} a$ and $q=\frac{b}{\sqrt{2m_1}}$ in the young inequality \eqref{E:=(A.3)}, we get \eqref{E:=(A.4)}. Secondly,  for all $0<m_1<1$ and for all positive numbers $a,\ b,$ and $c$, using \eqref{E:=(A.4)}, we get
\begin{equation*}
a\, b\, c\leq m_1 a^2+b^2 \frac{c^2}{4m_1}\leq m_1 a^2+m_1 b^4+\frac{c^4}{64m_1^3}\leq \left(a^2+b^4\right) m_1+\frac{c^4}{64m_1^3},
\end{equation*}
hence, we get \eqref{E:=(A.5)}.
\end{proof}
\noindent We will recall the relations between $p$ norms on $\mathbb{R}^m$.
\begin{Lemma} \label{L:=A.7}
For all $a=\left(a_1,\ldots,a_m\right)\in \mathbb{R}^m$, we have
\begin{equation}\label{E:=(A.6)}
\left(\sum_{j=1}^m\left|a_j\right|\right)^2\leq m \sum_{j=1}^m\left|a_j\right|^2\quad\text{and}\quad  \left(\sum_{j=1}^m\left|a_j\right|\right)^4\leq m^3 \sum_{j=1}^m\left|a_j\right|^4.
\end{equation}
\end{Lemma}

\end{document}